\documentclass{amsart}

\usepackage{bbm}
\usepackage{graphicx}
\usepackage{verbatim}

\graphicspath{{Fig/}}

\usepackage[latin1]{inputenc}

\newtheorem{prop}{Proposition}
\newtheorem{theorem}{Theorem}
\newtheorem{lemma}{Lemma}

\newtheorem{defi}{Definition}
\newtheorem{corollary}{Corollary}

\def\N{{\mathbb N}}
\def\R{{\mathbb R}}

\def\P{{\mathbb P}}
\def\E{{\mathbb E}}

\newcommand{\diff}{\mathop{}\mathopen{}\mathrm{d}}
\newcommand\croc[1]{\left\langle #1\right\rangle}
\newcommand\ind[1]{\mathbbm{1}_{\left\{#1\right\}}}

\def\cal{\mathcal}
\def\eps{\varepsilon}
\def\etal{{et al.}}

\title{\mbox{Analysis of Large Unreliable Stochastic Networks}}

\author{Wen Sun}\thanks{{\em Acknowledgments:} The first author's work is supported by a public grant overseen by the French National Research Agency (ANR) as part of the ``Investissements d'Avenir'' program (reference: ANR-10-LABX-0098). }
\email{Wen.Sun@inria.fr}
\author{Mathieu Feuillet}
\email{Mathieu.Feuillet@ssi.gouv.fr}
\address[M. Feuillet]{ANSSI,51 boulevard de la Tour-Maubourg, 75700 Paris 07 SP, France}
\author{Philippe  Robert}
\email{Philippe.Robert@inria.fr}
\urladdr{http://www-rocq.inria.fr/\string~robert}
\address[Ph. Robert,W. Sun]{INRIA Paris---Rocquencourt, Domaine de Voluceau, 78153 Le Chesnay, France}

\date{\today}

\begin{document}

\begin{abstract}
In this paper a stochastic model of a large distributed system where users' files are duplicated on unreliable data  servers is investigated. Due to a server breakdown,  a copy  of a file can be  lost, it can be retrieved if another copy of the same file is stored on other servers. In the case where no other copy of a given file is present in the network, it is definitively lost.  In order to have multiple copies of a given file, it is assumed that each server can devote a fraction  of its processing capacity to duplicate files on other servers to enhance the durability of the system.  

A simplified stochastic model of this network  is analyzed.  It is assumed that a copy of a given file is lost at some fixed rate and that the initial state is optimal: each file has the maximum number $d$ of copies located on the servers of the network.    The  capacity of duplication policy is used by the files with the  lowest number of copies.  Due to random losses, the state of the network is transient and all files will be eventually lost. As a consequence, a transient $d$-dimensional Markov process $(X(t))$ with a unique absorbing state describes the evolution this network. By taking  a scaling parameter $N$  related to the number of nodes of the network,  a scaling analysis of this process is developed.  The asymptotic behavior of $(X(t))$ is analyzed on time scales of the type $t\mapsto N^p t$ for  $0\leq p\leq d{-}1$.    The paper derives asymptotic results  on the decay of the network: Under  a stability assumption, the main results state  that the critical time scale for the decay of the system is given by $t\mapsto N^{d-1}t$. In particular  the duration of time after which a fixed fraction of files are lost is of the order of $N^{d-1}$.  When the  stability condition is not satisfied, i.e. when it is initially overloaded, it is shown that the state of the network converges to an interesting local equilibrium which is investigated. As a consequence it sheds some light on the role of the key parameters  $\lambda$, the duplication rate and $d$, the maximal number of copies,  in the design of these systems. The techniques used involve careful stochastic calculus for Poisson processes,  technical estimates and the proof of a stochastic averaging principle. 
\end{abstract}

\maketitle

\newpage 
\hrule

\vspace{-3mm}

\tableofcontents

\vspace{-10mm}

\hrule

\bigskip

\section{Introduction}

\subsection{Large Distributed Systems}
In this paper the problem of reliability of large distributed system is analyzed via mathematical models. A typical framework is a cloud computing environment where users' files are duplicated on several data servers.  When a server breaks down, all copies of files stored on this server are lost but they can be retrieved if copies of the same files are stored on other servers. In the case where no other copy of a given file is present in the network, it is definitively lost. Failures of disks occur naturally in this context, these events are quite rare but, given the large number of nodes of these large systems, this is not a negligible phenomenon at all at network scale.  For example, in a data center with 200~000 servers, in average five disks fail every day. See the extensive study Pinheiro \etal~\cite{Goog} in this domain at Google.  A natural consequence of these failures is the potential loss of some files if several servers holding copies of these files fail during a small time interval. For this reason this is a critical issue for companies deploying these large data centers.

\medskip
\noindent
{\sc Duplication Policies.} 
In order to maintain copies on distant servers, a fraction $\lambda$ of the bandwidth of each server has to be devoted to the duplication mechanism of its files to other servers. If, for a short period of time, several of the servers break down, it may happen that files will be lost for good just because all the available copies were on these servers and because a recovery procedure was not completed before the last copy disappeared. A second parameter of importance is $d$ the maximal number of copies of a given file in different servers. The general problem can then be presented as follows:  On the one hand, $d$ should be sufficiently large,  so that any file has a copy available on at least one server  at any time. On the other hand, the maximum number of copies for a given should not be too large, otherwise the necessary fraction of the server capacity for maintaining the number of copies would be very large and could impact other functions of the server. 

\subsection{Mathematical Models}
The natural critical parameters of such a distributed system with $N$ servers are the failure rate $\mu$ of servers, the bandwidth $\lambda$ allocated to duplication and the total number of files $F_N$. To design such a system, it is therefore desirable to have a duplication policy which maximizes the average number of files $\beta=F_N/N$ per server  and the first instant $T_N(\delta)$ when a fraction $\delta\in(0,1)$ of files is lost. The main goal of this paper is to give some insight on the role of these parameters through a simplified stochastic model. 

A lot of work has been done in computer science concerning the implementation of duplication algorithms. These systems are known as distributed hash tables (DHT). They play an important role in the development of some large scale distributed systems, see Rhea et al.~\cite{rhea-05} and Rowstron and Druschel~\cite{rowstron-01} for example. 

Curiously, except extensive simulations, little has been done to evaluate the performances of these algorithms. Simplified models using birth and death processes have been used. See Chun et al.~\cite{chun-06}, Picconi et al.~\cite{picconi-07} and Ramabhadran and Pasquale~\cite{ramabhadran-06}. In Feuillet and Robert~\cite{Feuillet-Robert}, a  mathematical model of the case of $d{=}2$ copies has been investigated.  In~\cite{Feuillet-Robert}, the main stochastic process of interest lives in dimension $1$ which  simplifies somewhat the analysis. As it will be seen, in our case, one has to investigate the more challenging problem of estimating some  transient characteristics of a $d{-}1$-dimensional  Markov process. 

To the best of our knowledge, there has not been any mathematical study investigating the dependence of the decay of the network, represented by the variable $T_N(\delta)$, with respect to the maximal number of copies $d$ and $\beta$ the average number of files per server. As it will be seen, even with a simplified model of the paper, the problem is already quite challenging.  One has to derive  estimates of transient characteristics of a transient $d$-dimensional Markov process on $\N^d$ with a reflection mechanism on the boundary of the state space. 

\medskip
\noindent
{\sc A Possible Mathematical Model.}
Without simplifying assumptions, a mathematical model could use a state descriptor $(Y_{j}(t), 1\leq j \leq F_N)$, where $Y_{j}(t)$ is  the subset  of  $\{1,\ldots,N\}$ of servers having a copy of file $j$ at time $t$. Note that the cardinality of $Y_i(t)$ is at most $d$ and that  file $i$ is lost if $Y_i(t)=\emptyset$. The transitions can be described as follows. 
\begin{enumerate}
\item  Loss: If, for $1\leq i\leq N$, node $i$ breaks down in state $(Y_j)$ then the value of $Y_j$ does not change if $i\not\in Y_j$ and, otherwise, $Y_j\mapsto Y_j{\setminus}\{i\}$. 
\item Duplication: if $1\leq i_1\not=i_2\leq N$ and $1\leq j\leq F_N$ are such that $|Y_j|<d$, $i_1\in Y_j$ and $i_2\not\in Y_j$, if the duplication policy at node $i_1$ does a copy of $j$ at $i_2$, then $Y_j\mapsto Y_{j}\cup\{i_2\}$ and the other coordinates are not affected by this change. Depending on the duplication policy at node $i_1$, the choice of the node $i_2$ and of file $j$ to copy may depend in a complicated way of the current state $(Y_j)$. 
\end{enumerate}
As it can be seen the state space is quite complicated and, moreover, its dimension  is growing with $N$ which is a difficulty to investigate the asymptotics for $N$ large. It does not seem to lead to a tractable mathematical model to study  for example the first instant when a fraction $\delta\in(0,1)$ of files are lost,
\[
\inf\left\{t\geq 0: \sum_{1}^{F_N} \ind{Y_j(t)=\emptyset}\geq \delta F_N\right\}. 
\]

\medskip
\noindent
{\sc Simplifying Assumptions.}
We present the mathematical model to be studied.  The model has been chosen so that the role of the parameter $d$ on the decay of the network can be investigated. To keep mathematics tractable,  simplifications for some of the other aspects of these systems have been done. We review the main features of our model and the assumptions we have done.
\begin{enumerate}
\item {\em Capacity for duplication.}\\
If there are $N$ servers and  each of them has an available bandwidth $\lambda$ to duplicate files, then the maximal capacity for duplication is $\lambda N$. One will assume that the duplication capacity can be used globally, i.e. the rate at which copies are created is $\lambda N$. 
\item {\em Duplication Policy.}\\
Moreover, the duplication capacity is used on the files with the {\em lowest number} of copies. The duplication capacity is in fact used at best, on the files that, potentially, are the most likely to be lost. 
\item {\em Failures.}\\
Any copy of a given file is lost at rate $\mu$. With this assumption, failures are more frequent but only a copy is lost at each event. In a more realistic setting, when a server breaks down, copies
of several different files are lost at the same time. 
\item {\em Topological Aspects.}\\
In practice, in DHT, servers are located on a logical ring and, in order to limit the communication overhead, the location of copies of a file owned by a given server $i$ are done at random  on a fixed subset $A_i$ of nodes, the leaf set of $i$. In our model, we assume that $A_i$ is the whole set of servers. 
\item {\em Statistical Assumptions.}\\
For mathematical convenience, the random variables used for the duration between two breakdowns of a server or of a duplication of a file are assumed to be exponentially distributed.
\end{enumerate}
 For this simplified model, the use of the total capacity of duplication is optimal, see items~(1) and~(2) below. Our results gives therefore an upper bound on the efficiency of duplication mechanisms in a general context. 

\medskip
\noindent
{\sc The Corresponding Markovian Model.}
With our assumptions, the state space can be embedded in a  fixed state space of dimension $d+1$.
If, for $0\leq i\leq d$ and $t\geq 0$, $X_i^N(t)$ is the number of files with $i$ copies, then the vector $X^N(t)=(X_0^N(t),X_1^N(t),\ldots,X_{d}^N(t))$ is a Markov process on $\N^{d+1}$. 

\medskip
\noindent
{\sc Transitions.}
The model starts initially with $F_N$ files, each of them having a maximal number of copies $d$, i.e. $X^N(0)=(0,0,\ldots,0,F_N)$. If $X^N(t)$ is in state $x{=}(x_i)\in\N^{d+1}$ and, for $0\leq i\leq d$, $e_i$ is the $i$th unit vector, there are two types of transitions for the Markov process. See Figure~\ref{fig2}.
\begin{enumerate}
\item  Loss: for $1\leq i \leq d$, $x\to x+e_{i-1}-e_{i}$.\\  A copy of a file with $i$ copies is lost  at rate $i x_i\mu$.
\item Duplication: for $1\leq i < d$, $x\to x-e_{i}+e_{i+1}$, $1\leq i<d$.\\ 
It occurs at rate $\lambda N$ under the condition $x_1=x_2=\cdots=x_{i-1}=0$, which means that there are no files with between $1$ and $i$ copies.
\end{enumerate}
Clearly enough, this system is transient, due to the random losses,  all files are eventually lost, the state $\emptyset=(F_N,0,\ldots,0)$ is an absorbing state. The aim of this paper is to describe the decay of the network, i.e. how the number $X_0^N(t)$ of lost files is increasing with respect to time. 

For fixed $F_N$ and $N$, this problem is related to the analysis of the transient behavior of a multi-dimensional Markov process. In our case, because of reflection on boundaries of $\N^{d+1}$ due to the duplication mechanism, the distribution of the evolution of the Markov process $(X_k(t))$ is not easy to study. For this reason, a scaling approach is used, with $N$ converging to infinity and $F_N$ being kept proportional to $N$. 

It will be assumed that the average number of files per server  $F_N/N$ converges to some $\beta>0$. For $\delta>0$,  the decay of the system can be represented by the random variable
\[
T_N(\delta)=\inf\left\{t\geq 0: \frac{X_0^N(t)}{N}\geq \delta \beta\right\}
\]
the time it take to have a fraction $\delta$ of the files lost.

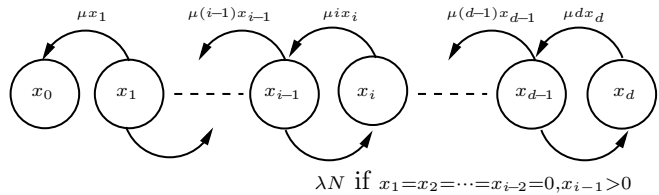
\begin{figure}
        \centering
\scalebox{0.4}{\includegraphics{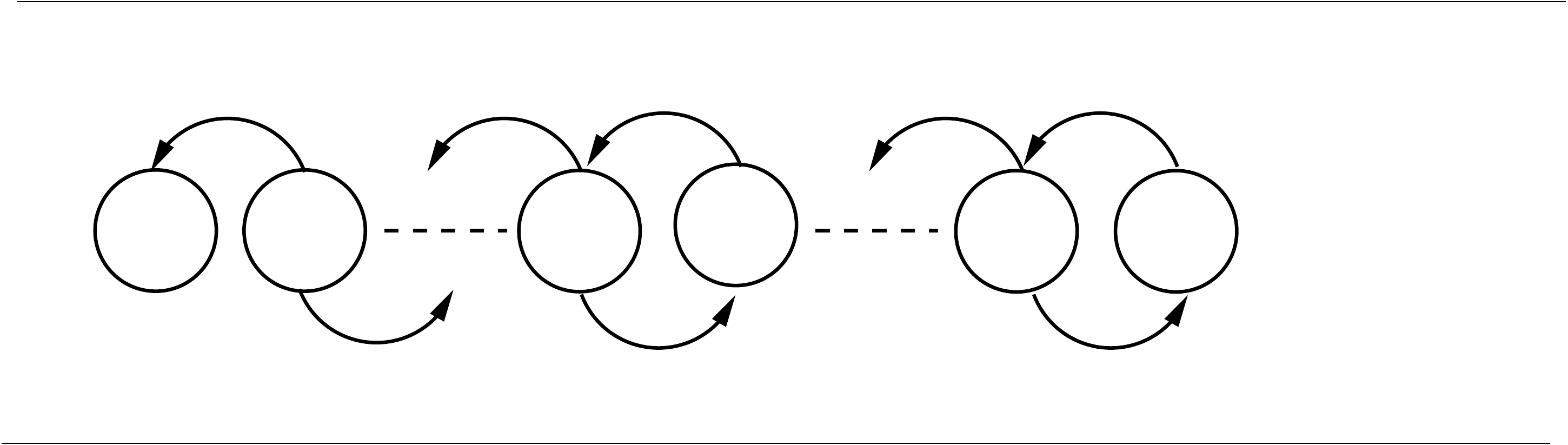}}
\put(-307,45){$\scriptstyle{x_0}$}
\put(-275,45){$\scriptstyle{x_{1}}$}
\put(-219,45){$\scriptstyle{x_{i\!{-}\!1}}$}
\put(-248,75){$\scriptscriptstyle{\mu(i\!{-}\!1)x_{i\!{-}\!1}}$}
\put(-197,75){$\scriptscriptstyle{\mu ix_i}$}
\put(-290,75){$\scriptscriptstyle{\mu x_1}$}
\put(-150,75){$\scriptscriptstyle{\mu (d\!{-}\!1) x_{d\!{-}\!1}}$}
\put(-106,75){$\scriptscriptstyle{\mu d x_d}$}
\put(-184,45){$\scriptstyle{x_{i}}$}
\put(-124,45){$\scriptstyle{x_{d\!{-}\!1}}$}
\put(-87,45){$\scriptstyle{x_{d}}$}
\put(-200,10){$ \scriptstyle{\lambda N}$ if $\scriptstyle{x_1{=}x_2{=}\cdots{=}x_{i\!{-}\!2}{=}0,x_{i{-}1}{>}0}$}
 \caption{Jump Rates for transfers of one unit between the coordinates of the  Markov Process $(X^N(t))$ in state $(x_0,x_1,\ldots,x_d)$}  \label{fig2}
\end{figure}

\subsection{Related Mathematical Models}\ 

\medskip
\noindent
{\sc Ehrenfest Urn Models.}
The Markov process $(X(t))$ can be seen as a particle system with $d+1$ boxes and any  particle in box $1\leq i\leq d$ moves to box $i-1$ at rate $\mu$. Box with index $0$ is a cemetery for particles.  A ``pushing'' process moves the particle the further on the left (box $0$ excluded) to the next box on its right at a high rate $\lambda N$. The model can be seen as a variation of the classical Ehrenfest urn model, see Karlin and McGregor~\cite{Karlin} and Diaconis \etal~\cite{Diaconis} for example. 

\medskip
\noindent
{\sc Polymerization Processes in Biology.}
 It turns out that this model has some similarities with stochastic processes representing polymerization processes of some biological models. The simplest model starts with a set of monomers (some proteins)  that can aggregate to form polymers. Due to random fluctuations within the cell,  a polymer of size $i$ and a monomer can produce a polymer of size $i+1$ at some fixed rate. In this context, as long as the size $i$ of the polymer is below some constant $i_0$, the polymer is not stable, it will lose monomers very quickly, at a high rate, it breaks into a polymer of size $i{-}1$ and a monomer. When the size is greater or equal to  $i_0$ (nucleation phase) the polymer is much more stable, it is assumed that it remains in this state.  Again due to the random fluctuations, all particles will end up in polymers of sizer greater than $i_0$. These ``large'' polymers  correspond to our box $0$ for the duplication process and the monomers are the equivalent of files with $d$ copies.     The {\em lag time} is the first instant when a positive fraction (half say)  of the monomers have been consumed into stable polymers. Note that it is analogous to our $T_N(1/2)$.  In this framework,  the fluctuations of the lag time have important consequences on biological processes. See Prigent \etal~\cite{Prion}, Xue et al.~\cite{Radford} and Szavits-Nossan et al.~\cite{Szavits} for example.

\subsection{Presentation of the Results}
The model starts with $F_N\sim\beta N$ files, all of them with the maximum number of copies $d$. The loss rate of a copy is $\mu$ and the duplication rate is $\lambda N$ but only for the files with the minimum number of copies. For $0\leq i\leq d$, $X_i^N(t)$ denotes the number of files with $i$ copies. 

It is first shown in Theorem~\ref{flth} that, as $N$ gets large, for the convergence of stochastic process
\begin{equation}\label{Soun}
\lim_{N\to+\infty} \left(\frac{X_k^N(t)}{N},0\leq k\leq d\right)=(x_k(t),0\leq k\leq d).
\end{equation}
The limit $(x_k(t),0\leq k\leq d)$ can be expressed as  the solution of a deterministic generalized Skorohod problem. See Section~\ref{Fluid} for the definition. 

\medskip
\noindent
{\sc Stable Case: $\lambda>d\mu\beta$.}\\
With the resolution of the generalized Skorohod problem, Proposition~\ref{charflprop} of Section~\ref{Fluid},  one proves that if $\lambda>d\mu\beta$ then the network is stable in the sense that the limiting process $(x_k(t),0\leq k\leq d)$ is constant and equal to $(0,\ldots,0,\beta)$. In other words, on the normal time scale,   the fraction of lost files is zero and, at the fluid level, all files have the maximal number of copies $d$.  

The key results of the stable case are Theorem~\ref{theodec} of Section~\ref{AverageSec} and  Theorem~\ref{theoCLT} of Section~\ref{CLT} .  These quite technical and delicate results rely on  careful stochastic calculus and various technical estimates related to the flows between coordinates of the process $(X_k^N(t))$. They are proved in several propositions of Section~\ref{AverageSec},  the important Proposition~\ref{lmtitg} in particular.  A stochastic averaging result completes the proof of these difficult convergence results. 

Theorem~\ref{theodec}  shows that the network is in fact beginning to lose files only on the time scale $t\mapsto N^{d-1} t$, i.e. that the convergence in distribution 
\begin{equation}\label{Nata}
\lim_{N\to+\infty} \left(\frac{X_0^N(N^{d-1}t)}{N}\right)=(\Phi(t))
\end{equation}
where $\Phi(t)$ is the unique solution $y\in[0,\beta]$ of the equation
\[
\left(1-\frac{y}{\beta}\right)^{\rho/d} e^y=\exp\left(-\lambda\frac{(d{-}1)!}{\rho^{d-1}} t\right).
\]
On this time scale, the fluid state of the network evolves from $(0,\ldots,0,\beta)$ to the absorbing state: $\Phi(0)=0$ and $\Phi(t)$ converges to $\beta$ as $t$ goes to infinity.

The second order fluctuations are described by the convergence in distribution
\[
\lim_{N\to+\infty} \left(\frac{X_0^N(N^{d-1} t)-N\Phi(t)}{\sqrt{N}}\right)=(W(t)),
\]
where $\Phi(t)$ is the solution of Equation~\eqref{decay} and  the process $(W(t))$ is the solution of  a stochastic differential equation, Relation~\eqref{SDECLT}.

\medskip
\noindent
{\sc Overloaded Case: $p\mu\beta<\lambda <(p+1)\mu\beta$ for some $2\leq p\leq d-1$.}\\
In this case,  the limiting process $(x_k(t),0\leq k\leq d)$ of Relation~\eqref{Soun} is not trivial, i.e. different from its initial state $(0,0,\ldots,0,\beta)$. Its explicit expression is given in Proposition~\ref{charflprop}. 
Moreover, it is shown that
\[
\lim_{t\to+\infty}(x_p(t),x_{p+1}(t))= \left((p+1)\beta -\frac{\lambda}{\mu}, \frac{\lambda}{\mu} - p \beta\right).
\]
This can be interpreted as follows: In the limit, on the normal time scale, at the fluid level all files have either $p$ or $p+1$ copies. The network exhibits therefore an interesting property of local equilibrium.

If one starts from this local equilibrium, it is shown that the system begins to lose files only the time scale $t\mapsto N^{p-1}t$. A result analogous to Relation~\eqref{Nata} is proved by Theorem~\ref{theoloc}, for the convergence in distribution, 
\[
\lim_{N\to+\infty} \left(\frac{X_0^N(N^{p-1}t)}{N},\frac{X_p^N(N^{p-1}t)}{N},\frac{X_{p+1}^N(N^{p-1}t)}{N}\right)=\left(\Phi_0(t),\Phi_p(t),\Phi_{p+1}(t)\right)
\]
holds, where $(\Phi_0(t),\Phi_p(t),\Phi_{p+1}(t))$ is deterministic, with the property that
\[
\lim_{t\to+\infty} \left(\Phi_0(t),\Phi_p(t),\Phi_{p+1}(t)\right)=\left(\beta-\frac{\rho}{p+1}, 0,\frac{\rho}{p+1} \right),
\]
i.e. asymptotically all files are either lost or have $p+1$ copies. 

These results give  the main phenomena concerning the evolution of a stable network  towards the absorbing state. It should be noted that we do not consider the special cases when the parameters satisfy the relation $\lambda{=}d\mu\beta$ for the following reason. When $\lambda{<}d\mu\beta$, the analysis involves a stochastic averaging principle with an underlying ergodic Markov process. See Section~\ref{StocAv} below. With equality $\lambda{=}d\mu\beta$, the corresponding Markov process is in fact null recurrent and proving a stochastic averaging principle in this context turns out to be more delicate. There are few examples in this domain to the best of our knowledge. See Khasminskii and Krylov~\cite{Khasminski} in the case of diffusions.  The same remark applies to  similar identities, like $\lambda=p\mu\beta$ for $1\leq p\leq d$.

\medskip
\noindent
{\sc Choice of Parameters.} 
As a consequence, the parameters $\beta$ and $d$ should be chosen so that $\lambda/(\beta \mu)>2$ and $d=\lfloor \lambda/(\beta \mu) \rfloor$ to maximize the time of decay of the network and at the same time to preserve the stability of the network. For $\delta\in(0,1)$ the variable $T_N(\delta)$, the first instant when a fraction $\delta$ of  files is lost,  is then of the order of $N^{d-1}$.

\medskip
\noindent
{\sc Outline of the Paper.} 
Section~\ref{ModSec} introduces the main notations and the stochastic evolution equations of the network. Section~\ref{Fluid} shows that the Markov process can be expressed as the solution of a generalized Skorohod problem, presented in Appendix~\ref{SkoSec}. A convergence result on the evolution of the network on the normal time scale is established and an explicit expression for the limiting process is provided. Section~\ref{AverageSec} investigates the decay of the network on the time scale $t\mapsto N^{d-1}t$ in the stable case. A central limit theorem on this time scale is established in Section~\ref{CLT}. The overloaded case is analyzed in Section~\ref{localsec}, the asymptotic evolution of the local equilibrium is studied on several time scales. 

\section{The Stochastic Model}\label{ModSec}
In this section we introduce the notations used throughout this paper as well as the statistical assumptions. The  stochastic  differential equations describing the evolution of the network are introduced. 
It is shown that, via a change of coordinates, the state descriptor of the process can be expressed as the solution of a generalized Skorohod problem. See Section~\ref{SkoSec}.  The convergence results at the normal time scale $t\mapsto t$ proved in the next section use this key property. 

A given file has a maximum of $d$ copies and each of them vanishes after an independent exponential time with rate rate $\mu$. A file with $0$ copy is lost for good. The recovery policy works as follows. The total capacity $\lambda N$ of the network is allocated to the files with the minimum number of copies. Consequently, if at a given time all non-lost files present have at least $k\geq 1$ copies and there are $x_k$ files with $k$ copies, then each of these is duplicated after  an independent exponential time with rate $\lambda N/x_k$. Initially it is assumed that there are $F_N$ files and that the network starts from the optimal state where each file has $d$ copies. 

For $0\leq k\leq d$, $X^N_k(t)$ denotes the number of files with $k$ copies at time $t$. The quantity $X^N_0(t)$ is the number of lost files at time $t$, the function  $t\mapsto X_0^N(t)$ is in particular non-decreasing. 

The conservation relation $X^N_0(t){+}X^N_1(t){+}\cdots{+}X^N_d(t){=}F_N$ gives that the stochastic process  $(X^N_0(t), X^N_1(t),\ldots,X^N_{d-1}(t))$ on $\N^d$ has the Markov property. Its $Q$-matrix $Q^N=(q^N(\cdot,\cdot))$ is given by
\begin{equation}\label{Qmatd}
\begin{cases}
\displaystyle \quad q^N(x,x-e_k+e_{k-1})&=\mu k x_k,\hfill  1\leq k\leq d-1,\\
\displaystyle \quad  q^N(x,x+e_{d-1})&=\mu d \left(F_N-x_0-x_1-\cdots -x_{d-1}\right),\\
\displaystyle \quad q^N(x,x+e_k-e_{k-1})&=\lambda N\ind{x_{k-1}>0, x_i=0,  1\leq i<k-1}, \hfill\qquad 2\leq
k\leq d-1,\\
\displaystyle  \quad q^N(x,x-e_{d-1})&=\lambda N\ind{x_{d-1}>0, x_i=0, 1\leq i<d-1},
\end{cases}
\end{equation}
where $e_k$ is the $k$th unit vector of $\N^{d}$. The first two relations come from the independence of losses of various copies of files, note that $F_N-x_0-x_1-\cdots -x_{d-1}$ is the number of files with $d$ copies. The last two equations are a consequence of the fact that the capacity  is  devoted  to  the  smallest  index   $k\geq  1$  such  that $x_k\not=0$. The coordinate $X_d^N(t)$ is of course given by 
\[
X_d^N(t)=F_N- X_0^N(t)-X_1^N(t)-\cdots-X_{d-1}^N(t),
\]
$X_d^N(t)$ is the  number  of   files  with  the  maximal   number  $d$  of  copies  at   time  $t$. The initial condition is such that 
$X_k^N(0)=0$ for $0\leq k\leq d-1$ and $X_d^N(0)=F_N\in\N$.  

\subsection*{Scaling Condition}
It is assumed that there exist some  $\beta>0$ and $\gamma\geq 0$ such that 
\begin{equation}\label{beta}
\lim_{N\to_\infty}\frac{F_N-N\beta}{\sqrt{N}}=\gamma.
\end{equation}

\subsection*{Equations of Evolution}\label{evosubsec}
To analyze the asymptotic behavior of the process $(X^N(t))$, it is convenient to introduce the  processes $(S^N(t)) {=} ((S_k^N(t),1 {\leq} k {\leq} d{-}1))$ and $(R^N(t)) {=} ((R_k^N(t),1{\leq} k{\leq} d{-}1))$. For $1\leq k\leq d-1$ and $t\geq 0$, $S_k^N(t)$ is the number of files with no more than $k$ copies at time $t$ and $R_k^N(t)$ is the local time at $0$ of the process $(S_k^N(t))$,
\[
S_k^N(t)=\sum_{i=1}^k X_i^N(t)\text{ and } R_k^N(t)=\int_0^t \ind{S_{k}^N(u)=0}\,\diff u.
\]
For any function  $h\in{\cal D}(\R_+,\R_+)$, i.e. $h$ is continuous on the right and has left limits on $\R_+$, one denotes by ${\cal N}_h$ denotes a point process on $\R_+$ defined as follows
\begin{equation}\label{Nh}
{\cal N}_h([0,t])=\int_0^t {\cal P}([0,h(u-)]\times \diff u)
\end{equation}
where $h(u{-})$ is the  left limit of $h$  at $u$ and ${\cal P}$ is a Poisson process in $\R_+^2$ whose intensity is the Lebesgue measure on $\R_+^2$. In particular if $h$ is deterministic, then ${\cal N}_h$ is a Poisson process 
with intensity $(h(t-))$. When several such processes ${\cal N}_{h}$ are used as below in the evolution equations, then the corresponding Poisson processes ${\cal P}$ used are assumed to be independent. The equations of evolution can then be written as, 
\[
\begin{cases}
 \diff S_{d-1}^N(t)={\cal N}_{d\mu(F_N-S^N_{d-1}-X_0^N)}(\diff t)-{\cal N}_{\mu S^N_{1}}(\diff t)
\\ \hspace{55mm} -\ind{S^N_{d-2}(t-)=0,S^N_{d-1}(t-)>0}{\cal N}_{\lambda N}(\diff t), \\
\\ 
 \diff S_k^N(t)={\cal N}_{(k+1)\mu X^N_{k+1}}(\diff t){-}{\cal N}_{\mu S^N_{1}}(\diff t)
{-}\ind{S^N_{k-1}(t-)=0,S^N_k(t-)>0}{\cal N}_{\lambda N}(\diff t),
\end{cases}
\]
for $1\leq k\leq d-2$, with the convention that $(S^N_0(t))\equiv 0$ is the null process and also that
$(R_0(t))=(t)$. By integrating and compensating  these equations, one gets that 
\begin{equation}\label{deveq}
S_k^N(t)=Z_{k}^N(t)-\lambda N(R^N_{k-1}(t)-R^N_{k}(t)),\quad 1\leq k\leq d-1,
\end{equation}
and the first coordinates $(X_0^N(t))$ satisfies the relation 
\begin{equation}\label{X0}
X_0^N(t)=\mu\int_0^t S_1^N(u)\,\diff u +U_0^N(t),
\end{equation}
with
\begin{align*}
Z_{d-1}^N(t)&=d\mu \int_0^t\left[F_N-S^N_{d-1}(u)-X^N_0(u)\right]\,\diff u-\mu\int_0^t S^N_{1}(u)\,\diff u+U_{d-1}^N(t)\\
Z_k^N(t)&=(k+1)\mu\int_0^t (S^N_{k+1}(u)-S^N_{k}(u))\,\diff u-\mu\int_0^t S^N_{1}(u)\,\diff u+U_k^N(t),
\end{align*}
for $1\leq k\leq d-2$, 
where the $(U^N(t))=(U_k^N(t), 1\leq k\leq d-1)$ are the martingales associated to the jumps of these
processes,
for $1\leq k\leq d-2$,
\begin{multline*}
U_k^N(t)=\int_0^t \left[{\cal N}_{\mu(k+1) X^N_{k+1}}(\diff u)-\mu(k+1) X^N_{k+1}(u)\,\diff u\right]\\
{-} \!\!\int_0^t  \!\!\left[{\cal N}_{\mu X^N_{1}}(\diff u){-}\mu X^N_{1}(u)\diff u\right]{-}
\int_0^t \ind{S^N_{k-1}(u{-}){=}0,S^N_{k}(u{-}){>}0}\left[{\cal N}_{\lambda N}(\diff u){-}\lambda N\diff u\right], 
\end{multline*}
and its increasing process is given by
\begin{multline}\label{incM}
\croc{U_k^N}(t)=\mu(k+1)\int_0^t  X^N_{k+1}(u)\,\diff u+\mu \int_0^t X^N_{1}(u)\,\diff u\\
+\lambda N\int_0^t \ind{S^N_{k-1}(u)=0,S^N_{k}(u)>0}\,\diff u.
\end{multline}
The martingales $(U^N_{0}(t))$ and $(U^N_{d-1}(t))$ have similar expressions,
\begin{multline*}
U_{d-1}^N(t)=\int_0^t \left[{\cal N}_{d\mu(F_N{-}S^N_{d-1}-X_0^N )}(\diff u)-d\mu(F_N{-}S^N_{d-1}(u){-}X_0^N(u))\diff u\right]\\
{-} \!\!\int_0^t \left[{\cal N}_{\mu X^N_{1}}(\diff u){-}\mu X^N_{1}(u)\diff u\right]{-}
 \!\!\int_0^t \!\!\ind{S^N_{d{-}2}(u{-}){=}0,S^N_{d-1}(u{-})>0}\left[{\cal N}_{\lambda N}(\diff u){-}\lambda N\diff u\right], 
\end{multline*}
with
\begin{multline*}
\croc{U_{d-1}^N}(t)=
d\mu \int_0^t\left(F_N{-}S^N_{d-1}(u){-}X_0^N(u)\right) \diff u\\
{+} \int_0^t \mu X^N_{1}(u)\diff u{+}\lambda N\int_0^t \ind{S^N_{d{-}2}(u{-}){=}0,S^N_{d-1}(u{-})>0}\diff u, 
\end{multline*}
and
\[
U_{0}^N(t)=\int_0^t \left[{\cal N}_{\mu S_1^N}(\diff u)-\mu S_1^N(u)\right]\,\diff u, \text{ with } \croc{U_{0}^N}(t)=\mu  \int_0^t S_1^N(u)\,\diff u.
\]
\subsection*{A Generalized Skorohod Problem Representation}
For $h=(h_i)$ an element of ${\cal D}(\R_+,\R^{d-1})$, $\eta>0$ and $F\in\N$, denote
\begin{equation}\label{Gdef}
\begin{cases}
&\displaystyle G_1(h,F,\eta)(t)=\mu\int_0^t (2h_{2}(u)-3h_{1}(u))\,\diff u-\eta t,\\
&\displaystyle G_k(h,F,\eta)(t)=\mu\int_0^t ((k+1)h_{k+1}{-}(k+1)h_k(u){-}h_1(u))\diff u,\, 1{<} k{<}d{-}1,\\
&\displaystyle G_{d-1}(h,F,\eta)(t)=d\mu \int_0^t\left[F{-} h_{d-1}(u){-}\mu\!\!\int_0^u\!\! h_1(v)\diff v\right]\,\diff u{-}\mu\int_0^t \!\!h_{1}(u)\diff u,
\end{cases}
\end{equation}
and $G(h,F,\eta)=(G_k(h,F,\eta),1\leq k\leq d-1)$. 
Equations~\eqref{deveq} and~\eqref{X0} give the relations
\begin{multline}\label{GSPevol}
S^N(t)=G\left(S^N,F_N,\lambda N\right)(t)+U^N(t)\\-d\mu \int_0^t U_0^N(u)\,\diff u \cdot e_{d-1}+\lambda N (I-P)R^N(t),
\end{multline}
where $P$ is the matrix $P=(P_{ij}, 1\leq i,j\leq d-1)$ whose non zero coefficients are the
$P_{i,i-1}=1$ for $2\leq i\leq d-1$. 

In other words, for a fixed $N$, the couple $(S^N, \lambda N R^N)$ is the solution of the generalized Skorohod
problem associated to the matrix $P$ and the functional 
\begin{equation}\label{Gbar}
\overline{G}: h\to G\left(h,F_N,\lambda N\right)+U^N-d\mu \int_0^\cdot U_0^N(u)\,\diff u \cdot e_{d-1}.
\end{equation}
See the appendix for the definition and a result of existence and uniqueness. 
\section{First Order Asymptotic Behavior}\label{Fluid}
In this section, the asymptotic behavior of the sequence of processes $(X_k^N(t)/N)$ at the ``normal'' time scale is investigated. As a consequence, it is shown that if $\lambda>\beta d\mu$ then the network is stable at the fluid level, i.e. the fraction of lost files  is $0$ at any time. Otherwise, a positive fraction of files is lost, an explicit expression for this quantity is provided. 

More precisely the convergence of the sequence of stochastic processes
\[
\left(\frac{X_k^N(t)}{N}, 0\leq k\leq d\right),
\]
is investigated. One first shows that this sequence is tight and the limit is identified as the solution of a deterministic generalized Skorohod problem. An explicit computation of this limit concludes the section. 

\subsection*{Tightness}
Due to Assumption~\eqref{beta}, there exists some  constant $C_0$ such that the relation $F_N\leq C_0 N$ holds  for all $N$.  Since
$0\leq X_k^N(t)\leq F_N$ for any $0\leq k\leq d-1$ and $t\geq 0$, Relation~\eqref{incM} gives the existence of a constant $C_1$ such that
\begin{equation}\label{Cdineq}
\E\left(U_{k}^N(t)^2\right)=\E\left(\croc{U_k^N}(t)\right)\leq C_1N t, \quad\forall 1\leq k< d-1 \text{ and } t\geq 0.
\end{equation}
with Doob's Inequality one gets that, for $1\leq k\leq d-1$ and $\eps>0$,
\[
\P\left(\sup_{0\leq s\leq t} \frac{U_k^N(s)}{N}\geq \eps \right)\leq \frac{1}{(\eps N)^2}\E(U_{k}^N(t)^2)
\leq \frac{C_1t}{\eps^2 N}
\]
shows that, for $0\leq k\leq d-1$, the martingale $(U_k^N(t)/N)$ converges in distribution to $0$ uniformly on
compact sets.

For  $T>0$,  $\delta>0$ and  for  $Z$  a function  in  the  space  $\mathcal{D}(\R_+,\R)$  of
c\`adl\`ag functions, i.e. right continuous functions  on $\R_+$ with left limits at every
point, define $w_{Z}(\delta)$ as the modulus  of continuity of the process $(Z(t))$ on the
interval $[0,T]$,
\begin{equation}\label{modul}
w_{Z}(\delta)=\sup_{0\leq s\leq t\leq T,\,|t-s|\leq \delta} |Z(t)-Z(s)|.
\end{equation}
By using again the relation $X_k^N(t)\leq C_0 N$ for all $N\in\N$, $1\leq k\leq d-1$
and $t\geq 0$, the above equations and the convergence of the martingales to $0$ give that,
for any $\eps>0$ and $\eta>0$, there exists $\delta>0$ such that
\[
\P(w_{S_k^N/N}(\delta)\geq \eta)\leq \eps,\quad 
\P(w_{X_0^N/N}(\delta)\geq \eta)\leq \eps,\quad  \forall N\text{ and }1\leq k\leq d-1.
\]
This implies that the sequence of stochastic processes 
\[
\left(\frac{X_0^N(t)}{N},\frac{S^N(t)}{N}\right)=\left(\frac{X_0^N(t)}{N}, \frac{S_k^N(t)}{N}, 1\leq k\leq d-1\right)
\]
is tight and that any of its limiting points is almost surely a continuous processes. See Billingsley~\cite{Billingsley} for example.

\subsection*{Convergence}
Let $(x_0(t),(s_k(t),1\leq k\leq d-1))$  denote a  limiting point  of the sequence  $(X_0^{N_p}(t)/N_p,S^{N_p}(t)/N_p)$
associated to  some non-decreasing subsequence  $(N_p)$. By choosing an appropriate probability space,
it can be assumed that the convergence holds almost surely. By Equation~\eqref{X0}, one gets that
\[
x_0(t)=\mu \int_0^t s_1(u)\,\diff u,
\]
From Definition~\eqref{Gbar} of the functional
$\overline{G}$ and by  convergence of the sequence of processes $(X_0^{N_p}(t)/N_p,S^{N_p}(t)/N_p)$ and of the martingale $(M^{N_p}/N_p)$ to $0$,  one gets that the convergence
\[
\lim_{p\to+\infty} \frac{1}{N_p}\overline{G}\left(S^{N_p},F_{N_p},\lambda N_p\right)=G\left(S,\beta,\lambda\right)
\]
holds uniformly on compact sets, where $G$ is defined by Relation~\eqref{Gdef}. As it has been seen in the previous section, Equation~\eqref{GSPevol}, the couple $(S^N/N,R^N/N)$ is the
solution of a classical Skorohod problem associated to the matrix $P$ introduced in Equation~\eqref{GSPevol} and the free process $(\overline{G}(S^{N_p},F_{N_p},\lambda N_p)/N_p$. By continuity of the solutions of a classical Skorohod problem, see  Proposition~5.11 of Robert~\cite{Robert} for example, one concludes that $(S^N/N,R^N)$ converges to the solution $(S,R)$ of the Skorohod problem associated to $P$ and  $h\mapsto G(h,\beta,\lambda)$. Hence $(S,R)$ is the unique solution of the {\em generalized} Skorohod problem for the matrix $P$ and the functional $h\to G(h,\beta,\lambda)$. The convergence of the sequence $(S^N/N,X_0^N/N)$ has been therefore established.

\begin{theorem}\label{flth}
If $S(t)=(s_k(t),1\leq k\leq d-1)$ is the unique solution of the generalized Skorohod problem
associated to the matrix $P=(\ind{(i,j)=(i,i-1)}, 1\leq i,j\leq d-1)$ and the functional
$h\mapsto G(h,\beta,\lambda)$ defined by Equation~\eqref{Gdef}, then  the  sequence of processes
\[
\left(\frac{X_k^{N}(t)}{N},0\leq k\leq d\right)
\]
converges in distribution uniformly on compact sets to $(x_k(t))$ defined by
\begin{align*}
x_0(t)&=\mu\int_0^t s_1(u)\,\diff u,\quad x_1(t)=s_1(t),\\
x_k(t)&=s_k(t)-s_{k-1}(t),\,2\leq k\leq d-1,\\
x_d(t)&=\beta-s_{d-1}(t)-x_0(t).
\end{align*}
\end{theorem}
If the limiting processes is uniquely determined as the solution of a Skorohod problem, it is not always easy to have an explicit representation of the solution of a Skorohod problem. The classical example of Jackson networks, see Chen and Mandelbaum~\cite{Chen}, shows that this is not always easy to have an explicit expression for the solutions of these problems in dimension greater than $2$. The linear topology of the network simplifies this question as the following proposition shows. 
\begin{prop}[Characterization of fluid limits]\label{charflprop}\ 
\begin{enumerate}
\item  $2\mu\beta<\lambda<d\mu\beta$.\\ Let $p=\lfloor \rho/\beta\rfloor$ with $\rho=\lambda/\mu$. the fluid limits $(s(t))=(s_1(t),\dots,s_d(t))$ of Theorem~\ref{flth} are defined as follows. 
There exist a sequence $(t_k)$,
\[
0=t_d<t_{d-1}<\dots<t_{p+1}<t_p=\infty,
\]
such that, for all $l=d-1,\ldots, p$ and for $t_{l+1} \leq t \leq t_{l}$,
\[
\begin{cases}
s_k(t) = 0,  \qquad 1\leq k \leq l{-}1,\\
s_{l}(t) = (l+1)\beta-\rho+\xi_{l,1} e^{-\mu t}+\sum_{i=l+2}^d\xi_{l,i}e^{-\mu i t},\\
s_k(t) = \beta\left(1- \sum_{i=k+1}^d\alpha_{k,i} e^{-i\mu t}\right),  \qquad l+1 \leq k \leq d,
\end{cases}
\]
where $\alpha_{d,d}=1$ and, for  $ j>l{+}1$
\[
\alpha_{l,j} =\frac{l{+}1}{l{+}1{-}j} \alpha_{l+1,j},\quad
\alpha_{l,l+1} = e^{(l+1)\mu t_l} \left( 1- \frac{\rho}{l\beta} - \sum_{k=l+2}^d \alpha_{l,k} e^{- k \mu t_l}  \right),
\]
\[
\xi_{l,j}=\frac{\beta (l+1)}{j-1}\alpha_{l+1,j},
\quad
\xi_{l,1}=-\left((l+1)\beta-\rho+\sum_{j=l+2}^d\xi_{l,j}e^{-\mu (j-1)t_{l+1}}\right),
\]
with $\alpha_{l,l}=0$ and $t_l$ is the unique solution of  $s_{l}(t)={\lambda}/{(l\mu)}$.
\item[]
\item   $\lambda>d\mu\beta$.\\ For all $t\geq 0$, $(x_1(t),\dots,x_d(t))=(0,\dots,0,\beta)$. 
\end{enumerate}
\end{prop}
\begin{proof}
The vector $(s_k(t))$  is  solution of the following equation:
\begin{align*}
s_k(t) &= \mu(k+1) \int_0^t \left( s_{k+1}(u) - s_k(u) \right) \diff u - \mu \int_0^t s_1(u)\diff u - \lambda (r_{k-1}(t)-r_{k}(t)),\\
s_d(t) &=\beta  - \int_0^t \mu s_1(u) \diff u,
\end{align*}
where the $(r_k(t))$ are the reflection processes such that
$$
\int_0^t s_k(u) \diff r_k(u) = 0.
$$

By uniqueness of the solution of a generalized Skorohod problem given by Proposition~\ref{GSPprop} of the Appendix, it is enough to exhibit a solution to the above equations. 

We assume the conditions of the case~(1) of the proposition. 
We will prove in fact that there exists $t_d=0< t_{d-1} <t_{d-2} < \dots < t_p < t_{p-1}=+\infty$ such that,
for all $p\leq l \leq d-1$ and $t_{l+1}<t<t_l$, the $s_k$ and the $t_k$ have the
following equations:
\begin{equation}
\begin{cases}
\begin{aligned}
&s_k(t) = 0, && r_k(t)=t, \qquad 1\leq k \leq l-2,\\
&s_{l-1}(t) = 0,  &&\dot{r}_{l-1}(t)= 1 - l\mu/\lambda s_{l}(t), \\
&\dot{s}_l(t) = \mu(l+1) s_{l+1}(t) - \mu s_l(t) - \lambda,  &&\dot{r}_l(t), = 0 \\
&\dot{s}_k(t) = \mu(k+1) (s_{k+1}(t)-s_k(t)), && \dot{r}_k(t)=0, \qquad l+1\leq k \leq d-1,\\
&s_d(t) = \beta,  &&r_d(t) = 0.
\end{aligned}
\end{cases}
\label{eq:system}
\end{equation}
The $t_k$ are defined such that $s_k(t_k) = \lambda/(\mu k)$.

We start with the case $d-1$. It is easy to check
that  ($(s_k)$, $(r_k)$) defined by the following equations is the solution of
the generalized Skorohod problem,
\[
\begin{cases}
&s_k(t)=0,\quad  r_k(t)=t,\qquad \text{for } 1\leq k\leq d-3,\\
&\displaystyle s_{d-2}=0, \qquad r_{d-2}(t) = t - \frac{(d-1)\mu}{\lambda} \int_0^ts_{d-1}(u)\diff u, \\
&s_{d-1}(t)= \left(d \beta -{\lambda}/{\mu}\right)(1-e^{-\mu t}),\quad r_{d-1}(t) = 0.
\end{cases}
\]
This is valid for all $0\leq t < t_{d-1}$ with
$$
  t_{d-1} = \frac{1}{\mu} \log\left(\frac{d\beta - \rho}{d \beta  - d\rho/(d-1)} \right).
$$

Now, we proceed by using a recursion. Assume that there exists $l>p$ such that
the system of equation \eqref{eq:system} is verified until $t_l$. Moreover, we assume
that, for all $k \geq  l$ and  $t_k\leq t\leq t_{k-1}$,
$$
s_k(t) = \beta\left(1-\sum_{i=k+1}^d \alpha_{k,i} e^{-\mu i t}\right);
$$
and
$$
s_{k-1}(t) = k\beta-\rho+\xi_{k-1,1} e^{-\mu t}+\sum_{i=k+1}^d\xi_{k-1,i}e^{-\mu i t};
$$
and $t_{k-1}$ is the only solution of
$$
s_{k-1}(t_{k-1})=\frac{\lambda}{(k-1)\mu}.
$$

We define
$$
\alpha_{l,i} =\frac{l+1}{l+1-i} \alpha_{l+1,i} ,\text{ for } d\ge i>l+1,
$$
$$
\alpha_{l,l+1} = e^{(l+1)\mu t_l} \left( 1-\frac{\rho}{l\beta} - \sum_{k=l+2}^d \alpha_{l,k} e^{- k \mu t_l}  \right),
$$
$$
\xi_{l-1,i}=\frac{\beta l}{i-1}\alpha_{l,i}, \text{ for }d\ge i>l+1,
$$
and
$$
\xi_{l-1,1}=-\left(l\beta-\rho+\sum_{i=l+1}^d\xi_{l-1,i}e^{-\mu (i-1)t_l}\right).
$$
It is easy to check that $s_{l}$ is then solution of the equation
\[
\dot{s}_l = \mu(l+1) (s_{l+1}(t)-s_l(t))
\]
when $t\ge t_l$ and $s_{l-1}$ is the solution of the equation
$$
\dot{s}_{l-1}(t) = \mu l s_{l}(t) - \mu s_{l-1}(t) - \lambda,
$$
when $t_{l-1}\ge t\ge t_{l}$, $t_{l-1}$ is the solution of the equation
$$
s_{l-1}(t_{l-1})=\frac{\lambda}{(l-1)\mu}.
$$
The recursion is proved and therefore the assertion of case~(1) of the proposition. 

Concerning the case~(2), one has only to check that  the couple
\[
\begin{cases}
(x_1(t),\dots,x_d(t))\stackrel{\text{def.}}{=}(0,\dots,0,\beta)\\
(r_1(t),\ldots,r_{d-2}(t),r_{d-1}(t),r_{d}(t))\stackrel{\text{def.}}{=}\left(t,\ldots,t,\left(1-{d\mu\beta}/{\lambda}\right)t,0\right)
\end{cases}
\]
is indeed the solution of the generalized Skorohod problem.
\end{proof}
The following corollary shows that in the overloaded cases, asymptotically, there is an equilibrium where most of files will have either $p$ or $p+1$ copies for some convenient $p$. This situation is investigated in Section~\ref{localsec}. 
\begin{corollary}[Stable Fluid State of the Overloaded System]\label{fluid2}
In the case~(1) of Proposition~\ref{charflprop}, then 
\[
\lim_{t\to+\infty}(x_p(t),x_{p+1}(t))= ((p+1)\beta -\rho, \rho - p \beta),
\]
and $x_k(t) \to 0$ as $t\to+\infty$ for all $1\leq k\leq d$, $k\not\in \{p,p+1\} $.
\end{corollary}

\vspace{-7mm}

\begin{figure}[ht]
\centering
\scalebox{1}{\includegraphics[height=8cm]{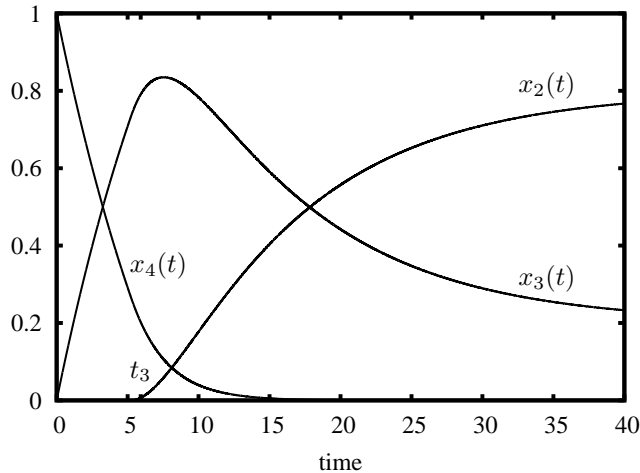}}
\put(-212,60){$t_3$}
\put(-212,100){${x_4(t)}$}
\put(-65,95){${x_3(t)}$}
\put(-65,168){${x_2(t)}$}
 \caption{Fluid Limits of an Overloaded Network, $2\beta\mu{<}\lambda{<}3\beta\mu$ with $d=4$, $\mu=0.1$, $\lambda=0.22$, $\beta=1$. In this case $t_3=5.23$.}  \label{figfluid}
\end{figure}

\subsection*{Example}
To illustrate the results of Proposition~\ref{charflprop}, one considers the case $d=4$ and with the condition $2\beta\mu<\lambda<3\beta\mu$. One easily gets that
\[
t_3=  \frac{1}{\mu}\log\left(\frac{3}{4}\cdot\frac{4\beta-\rho}{3\beta-\rho}\right)
\]
and
\[
s_3(t)=
\begin{cases}
 (4\beta-\rho)(1-e^{-\mu t})&\text{ if } t\leq t_3,\\
\displaystyle\beta-\frac{27}{256} \frac{(4\beta-\rho)^4}{(3\beta-\rho)^3}e^{-4\mu t}&\text{ if } t\geq t_3,
\end{cases}
\]
and for $t<t_3$, $s_2(t)=0$ and if $t>t_3$,
\[
s_2(t)= (3\beta-\rho) - (4\beta-\rho)e^{-\mu t}  -\frac{27}{256} \frac{(\rho-4\beta)^4}{(\rho-3\beta)^3} e^{-4\mu t}.
\]
Finally $s_1(t)=x_0(t)=0$, for all $t\geq 0$. Figure~\ref{figfluid} presents a case with $d=4$ and where, asymptotically, a local equilibrium holds:  files have either $2$ or $3$ copies as $t$ goes to infinity.

\section{Evolution of  Stable Network}\label{AverageSec}
In this section, the asymptotic properties of the sequence of processes 
\[
(X^N(t))=(X_0^N(t),X_1^N(t),\ldots,X_{d-1}^N(t))
\]
are investigated under the condition $\rho=\lambda/\mu>d\beta$ and with the initial state $X^N(0)=(0,\ldots,0,F_N)$. Section~\ref{Fluid} has shown that, in this case, the system is stable at the first order, i.e. that the fraction of lost files is $0$. This does not change the fact that the system is transient with one absorbing state $(F_N,0,\ldots,0)$.  The purpose of this section is of showing that the decay of this networks occurs on the time scale $t\mapsto N^{d-1}t$.

The section is organized as follows, preliminary results, Lemma~\ref{propdom} and Proposition~\ref{lemLGN} partially based on couplings show that the coordinates in the middle, i.e. with index between $1$ and $d-1$, cannot be very large on any time scale. In a second step,  Proposition~\ref{prop1} and Proposition~\ref{lmtitg} show that the flows between the coordinates of the Markov process are ``small''. Proposition~\ref{lmtitg} is the crucial technical result of this section. Finally, the asymptotic study of a random measure on $\N\times\R_+$ gives the last element to establish the main result, Theorem~\ref{theodec}, on the evolution of the network on the time scale $t\mapsto N^{d-1}t$. 

\medskip
\noindent
{\bf Stochastic Differential Equations}\\
The SDE satisfied by the process $(X_k^N(t))$ are recalled. As before, if $1\leq k\leq d$, $S_k^N(t)=X_1^N(t)+\cdots+X_{k}^N(t)$ and  the convention that $S_0^N\equiv 0$ and $S_{-1}^N\equiv -1$, then
\begin{align}
&X_0^N(t)=\mu\int_0^t X_1^N(u)\,\diff u +M_0^N(t),\label{SDE11}\\
&X_{k}^N(t)=\mu (k{+}1)\int_0^t X_{k+1}^N(u)\,\diff u-\mu k\int_0^t X_{k}^N(u)\,\diff u\label{SDE1k}\\
&+\lambda N\int_0^t \ind{S_{k-2}^N(u)=0,X_{k-1}^N(u)>0}\,\diff u-\lambda N\int_0^t \ind{S_{k-1}^N(u)=0,X_{k}^N(u)>0}\,\diff u\notag \\
&+M_{k}^N(t),\quad \text{ for } 1\leq k\leq d-1,\notag
\end{align}
where, for $0\leq k\leq d-1$, $(M_k^N(t))$ is a square integrable martingale whose previsible increasing process is given by 
\begin{align}
&\croc{M_0^N}(t)=\mu\int_0^t X_1^N(u)\,\diff u,\label{SDEV1}\\ 
&\croc{M_k^N}(t)=\mu (k{+}1)\int_0^t X_{k+1}^N(u)\,\diff u+\mu k\int_0^t X_{k}^N(u)\,\diff u\label{SDEVk}\\
&+\lambda N\int_0^t \ind{S_{k-2}^N(u)=0,X_{k-1}^N(u)>0}\,\diff u+\lambda N\int_0^t \ind{S_{k-1}^N(u)=0,X_{k}^N(u)>0}\,\diff u\notag
\end{align}
\subsection{Some Technical Results} 
We start with two preliminary results on a coupling of the network. 
\begin{lemma}\label{lemLGN}
If $(L(t))$ is the process of the number of customer of an $M/M/1$ queue with arrival rate $\alpha$ and service rate $\gamma>\alpha$ and with initial condition $L(0)=x_0\in\N$  then, for the convergence in distribution of continuous processes,
\[
\lim_{N\to+\infty} \left(\frac{1}{N}\int_0^{Nt} L(u)\,\diff u\right)=\left(\frac{\alpha}{\gamma-\alpha}t\right).
\]
\end{lemma}
\begin{proof}
The proof is standard, see the end of Proof of Proposition~9.14 page~272 of Robert~\cite{Robert} for example. 
\end{proof}
The next proposition presents an important property of the network. Roughly speaking it states that the $F_N$ files have either $0$ or $d$ copies on the time scale $t\mapsto N^{d-1}t$.  Coordinates with index between $1$ and $d-1$ of the vector $(X^N(t))$ remain small.
\begin{prop}[Coupling]\label{propdom}
Under the condition $d\beta\mu<\lambda$ and with the initial state  $X^N(0)=(0,\ldots,0,F_N)$, one can find a probabilistic space so that  the relation
\[
(d-1)X_1^N(t)+(d-2)X_2^N(t)+\cdots+X_{d-1}^N(t)\leq L_{0}(N t),\qquad \forall t>0,
\]
holds, where the vector $(X_k^N(t), 1\leq k\leq d{-}1)$ has the same distribution as the state of our network and  $(L_{0}(t))$ is the process of the number of customers of  an  $M/M/1$ queue with arrival rate $d\mu\beta_0$ and service rate $\lambda$ and  with the initial condition $L_{0}(0)=0$ for some $\beta_0$ satisfying $d\mu\beta_0<\lambda$. 

For all $i=1$, $2$, \ldots, $d{-}1$ and $\alpha>0$ then,  for the convergence in distribution of continuous processes,  the relation
\begin{equation}\label{serv13}
\lim_{N\to\infty}\left(\frac{X_{i}^N(N^{d-1}t)}{N^\alpha}\right)=0
\end{equation}
holds. 
\end{prop}
\begin{proof}
The existence of $N_0$ and $\beta_0$ such that  $d\mu\beta_0<\lambda$ and $F_N\leq \beta_0 N$ for $N\geq N_0$ is clear. 
Define
\[
Z^N(t)=(d-1)X_1^N(t)+(d-2)X_2^N(t)+\cdots+X_{d-1}^N(t),
\]
then the possible jumps of $(Z^N(t))$ are either $1$, $-1$ or ${-}(d{-}1)$.  If $X^N(t)=(x_k)$,  jumps of size $1$ occur at rate
$\mu[2x_2+\cdots+(d-1)x_{d-1}+dx_d]\leq \mu dF_N\leq \mu d \beta_0 N$. Similarly jumps of size $-1$ occurs at rate $\lambda N$ provided that $Z^N(t)\not=0$. A simple coupling gives therefore  that $(Z_N(t))$ is upper bounded by an $M/M/1$ queue with service rate $\lambda N$ and arrival rate $\mu d \beta_0 N$. The first part of the proposition is proved. 

By ergodicity of the $M/M/1$ process $(L_0(t))$, one has, for the convergence in distribution,
\begin{equation}\label{eq1}
\lim_{N\to+\infty} \left(\frac{L_0(N^\kappa t)}{N^\alpha}\right)=0
\end{equation}
for all $\kappa>0$ and $\alpha>0$.  Indeed, if 
\[
T_K=\inf\{s\geq 0: L_0(s)\geq K\},
\]
then, if $\delta=d\mu\beta_0/\lambda$,  the random variable $\delta^K T_K$ is converging in distribution to an exponential random variable as $K$ goes to infinity. 
See Proposition~5.11 page~119 of  Robert~\cite{Robert} for example. 

For $T>0$ and $\eps>0$, one has
\[
\P\left(\sup_{0\leq s\leq T} \frac{L_0(N^\kappa t)}{N^\alpha} \geq \eps \right)
=\P\left(T_{\lceil \eps N^\alpha\rceil}\leq N^\kappa T\right)
\]
and since $\delta<1$, this last term is converging to $0$ as $N$ goes to infinity. Convergence~\eqref{eq1} has therefore been proved.  One concludes that the sequence of processes $(Z^N(N^{d-1}t)/N^\alpha)$ converges in distribution to $0$.  The proposition is proved. 
\end{proof}

\begin{prop}\label{prop1}
Under the condition $d\beta\mu<\lambda$ and  if $X^N(0)=(0,\ldots,0,F_N)$,
then,  for $1\leq k\leq d-1$ and any  $\gamma>0$.    one has
\begin{equation}\label{eqsaux2}
\lim_{N\to\infty}\left(\frac{1}{N^{k+\gamma}}\int_0^{N^{d{-}1} t}X_k^N(u)\, \diff u\right)=0.
\end{equation}
for the convergence in distribution of continuous processes, and, for any $t\geq 0$,
\begin{equation}\label{eqsaux3}
\lim_{N\to\infty}\frac{1}{N^{k+\gamma}}\int_0^{N^{d{-}1} t}\E\left(X_k^N(u)\right)\, \diff u=0.
\end{equation}
\end{prop}
\begin{proof}
One proceeds by induction on $1\leq k\leq d-1$.  Let $k=1$, if $t\geq 0$,
Equation~\eqref{SDE11} gives the relation
\begin{equation}\label{eqsaux}
\mu\frac{1}{N^{1+\gamma}}\int_0^{N^{d-1}t}X_1^N(u)\,\diff u=\frac{X_0^N(N^{d-1} t)}{N^{1+\gamma}}-\frac{M_0^N(N^{d-1}t)}{N^{1+\gamma}},
\end{equation}
by Doob's Inequality and Equation~\eqref{SDEV1}, for $\eps>0$,
\begin{multline*}
\P\left(\sup_{0\le u\le t}\frac{|M_0^N(N^{d-1}u)|}{N^{1+\gamma}}\ge \eps\right)
\\\le \frac{1}{\eps^2}\mu \frac{1}{N^{2+2\gamma}}\E\left(\int_0^{N^{d-1}t}X_1^N(N^{d-1}u)\,\diff u\right)
=  \frac{1}{\eps^2} \E\left(\frac{X_0^N(N^{d-1} t)}{N^{2+2\gamma}}\right)
\end{multline*}
by using again Equation~\eqref{eqsaux}. The variable $X_0^N$ being upper bounded by $F_N$, the last equality shows that Convergence~\eqref{eqsaux3} holds in this case.  Additionally one gets that the martingale term of Equation~\eqref{eqsaux} vanishes at infinity. The convergence~\eqref{eqsaux2} is therefore proved. Induction assumption is thus true for $k=1$. 

Assume by that   induction assumption holds up to index $k{<}d{-}1$. Equation~\eqref{SDE1k} gives 
\begin{multline}\label{eqsaux1}
\frac{\mu (k{+}1)}{N^{k+1+\gamma}} \int_0^{N^{d-1}t} X_{k+1}^N(u)\,\diff u=\frac{X_{k}^N(N^{d-1}t)}{N^{k+1+\gamma}}  +\frac{\mu k}{N^{k+1+\gamma}} \int_0^{N^{d-1}t} X_{k}^N(u)\,\diff u
\\-\frac{\lambda }{N^{k+\gamma}}\int_0^{N^{d-1}t} \ind{S_{k-2}^N(u)=0,X_{k-1}^N(u)>0}\,\diff u\\+\frac{\lambda}{N^{k+\gamma}}\int_0^{N^{d-1}t} \ind{S_{k-1}^N(u)=0,X_{k}^N(u)>0}\,\diff u -\frac{M_{k}^N({N^{d-1}t})}{N^{k+1+\gamma}}.
\end{multline}
Note that, for $i=k-1$, $k$,
\[
\frac{\lambda }{N^{k+\gamma}}\int_0^{N^{d-1}t} \ind{S_{i-1}^N(u)=0,X_{i}^N(u)>0}\,\diff u\leq 
\frac{\lambda }{N^{k+\gamma}}\int_0^{N^{d-1}t} X_{i}^N(u)\,\diff u.
\]
By integrating Equation~\eqref{eqsaux1} and  using the induction assumption, one obtains that Convergence~\eqref{eqsaux3} holds for $k+1$. Back  to Equation~\eqref{eqsaux1}, by induction again, 
the first four terms of the right hand side of Equation~\eqref{eqsaux1} converges to $0$ and the martingale term vanishes since the expected value of its previsible increasing process is converging to $0$ by Relation~\eqref{SDEVk} and Convergence~\eqref{eqsaux3} which has been established.  
\end{proof}
\begin{prop}\label{lmtitg}
Under the condition $d\beta\mu<\lambda$ and  if $X^N(0)=(0,\ldots,0,F_N)$,
then  the relations,  for  $ 1\leq k\leq d{-}2$,
\begin{equation}\label{id1}
\lim_{N\to\infty}\left(\frac{1}{N^{k+1/2}} \int_0^{N^{d-1}t} 
\left[(k+1) \mu X_{k+1}^N(u)-\lambda N  X_k^N(u)\right]\,\diff u\right)=0
\end{equation}
holds for  the convergence in distribution of continuous processes. 
\end{prop}
\begin{proof}
One proves Convergence~\eqref{id1} for $1\leq k\leq d-2$.
With the evolution equation~\eqref{SDE1k} and the same notation~\eqref{Nh} as in Section~\ref{evosubsec} for the Poisson processes, for any function $f:\N\to\R_+$,  one has 
\begin{multline*}
f\left(X_{k}^N(t)\right){=}f\left(X_{k}^N(0)\right)
{+}\int_0^{t}\left[\rule{0mm}{4mm}f\left(X_k^N(u{-}){+}1\right){-}f\left(X_k^N(u{-})\right)\right]\,{\cal N}_{\mu (k{+}1)X_{k+1}^N}(\diff u)\\
\int_0^{t}\left[\rule{0mm}{4mm}f\left(X_k^N(u{-}){-}1\right){-}f\left(X_k^N(u{-})\right)\right]\,{\cal N}_{\mu k X_{k}^N}(\diff u)\\
{+}\int_0^{t} \left[\rule{0mm}{4mm}f\left(X_k^N(u{-}){+}1\right){-}f\left(X_k^N(u{-})\right)\right]\ind{S_{k-2}^N(u{-})=0,X_{k-1}^N(u{-})>0}\,{\cal N}_{\lambda N}(\diff u)\\
{+}\int_0^{t} \left[\rule{0mm}{4mm}f\left(X_k^N(u{-}){-}1\right){-}f\left(X_k^N(u{-})\right)\right]\ind{S_{k-1}^N(u{-})=0,X_{k}^N(u{-})>0}\,{\cal N}_{\lambda N}(\diff u).
\end{multline*}
By taking $f(x){=}x^2$ and by compensating the Poisson processes, one gets the relation
\begin{multline}\label{eqsaux5}
X_{k}^N(N^{d-1}t)^2=\mu (k{+}1)\int_0^{N^{d-1}t}( 2 X_{k}^N(u)+1)X_{k+1}^N(u)\,\diff u\\+\mu k\int_0^{N^{d-1}t}( -2X_{k}^N(u)+1) X_{k}^N(u)\,\diff u\\
+\lambda N\int_0^{N^{d-1}t} ( 2 X_{k}^N(u)+1)\ind{S_{k-2}^N(u)=0,X_{k-1}^N(u)>0}\,\diff u\\+\lambda N\int_0^{N^{d-1}t} ( -2 X_{k}^N(u)+1) \ind{S_{k-1}^N(u)=0,X_{k}^N(u)>0}\,\diff u
+M_{k,2}^N(t).
\end{multline}
The process $(M_{k,2}^N(t))$ is a martingale with a previsible increasing process given by
\begin{multline}\label{eqsaux6}
\croc{M_{k,2}^N}(t)= \mu (k{+}1)\int_0^{N^{d-1}t}( 2 X_{k}^N(u)+1)^2X_{k+1}^N(u)\,\diff u\\+\mu k\int_0^{N^{d-1}t}( -2X_{k+1}^N(u)+1)^2 X_{k}^N(u)\,\diff u\\
+\lambda N\int_0^{N^{d-1}t} ( 2 X_{k}^N(u)+1)^2\ind{S_{k-2}^N(u)=0,X_{k-1}^N(u)>0}\,\diff u\\+\lambda N\int_0^{N^{d-1}t} ( -2 X_{k}^N(u)+1)^2 \ind{S_{k-1}^N(u)=0,X_{k}^N(u)>0}\,\diff u.
\end{multline}
By adding up Equations~\eqref{SDE1k} and~\eqref{eqsaux5}, after some straightforward calculations,  one gets
\begin{multline}\label{eqsaux4}
X_{k}^N(N^{d-1}t){+}X_{k}^N(N^{d-1}t)^2=2\int_0^{N^{d-1}t} \left[\mu (k{+}1) X_{k+1}^N(u){-}\lambda N X_{k}^N(u)\right]\,\diff u\\
+2\mu(k{+}1) \int_0^{N^{d-1}t} \!\!\! X_k^N(u)X_{k+1}^N(u)\,\diff u -2\mu k\int_0^{N^{d-1}} X_k^N(u)^2\,\diff u
\\+2\lambda N\int_0^{N^{d-1}t}\!\!\! (X_{k}^N(u)+1)\ind{S_{k-2}^N(u)=0,X_{k-1}^N(u)>0}\,\diff u\\+2\lambda N\int_0^{N^{d-1}t}  X_{k}^N(u)\ind{S_{k-1}^N(u)>0,X_{k}^N(u)>0}\,\diff u
+M_{k}^N(t)+M_{k,2}^N(t).
\end{multline}
It will be shown that, when this relation is scaled by the factor $N^{k+1/2}$, except the first integral in the right hand side, all terms of this identity vanish as $N$ gets large.  
The proposition will be then  proved. 

For the terms of the left hand side this is clear. 
For $1{\leq} k{\leq} d{-}2$  and $j{\in}\{k,k{+}1\}$, the relation 
\begin{multline*}
\frac{1}{N^{k+1/2}}\int_0^{N^{d-1}t}X_k^N(u)X_{j}^N(u)\,\diff u
\\ \leq \sup_{0\leq u\leq t}\left( \frac{X_{j}^N(N^{d-1}u)}{N^{1/4}} \right)\frac{1}{N^{k+1/4}}\int_0^{N^{d-1}t}X_k^N(u)\,\diff u,
\end{multline*}
and Propositions~\ref{propdom} and~\ref{prop1} show that  the second term of the right hand side of Equation~\eqref{eqsaux4} scaled by $N^{k+1/2}$ vanishes for the convergence of processes when $N$ gets large. By using the inequalities
\begin{align*}
\int_0^{N^{d-1}t}  X_{k}^N(u)\ind{S_{k-2}^N(u)=0,X_{k-1}^N(u)>0}\,\diff u& \leq \int_0^{N^{d-1}t}  X_{k}^N(u)X_{k-1}^N(u)\,\diff u,
\intertext{ and } 
\int_0^{N^{d-1}t}  X_{k}^N(u)\ind{S_{k-1}^N(u)>0,X_{k}^N(u)>0}\,\diff u&\leq\sum_{i=1}^{k-1} \int_0^{N^{d-1}t}  X_{i}^N(u)X_{k}^N(u)\,\diff u,
\end{align*}
the same property can be established in a similar way  for the third, fourth and fifth terms. 

By using Equations~\eqref{SDEVk} and~\eqref{eqsaux6} and similar methods one gets  that for any $t\geq 0$,
\[
\lim_{N\to+\infty} \frac{\E\left(\croc{M_{k}^N}(N^{d-1}t)\right)}{N^{2k+1}}=0,\quad
\lim_{N\to+\infty} \frac{\E\left(\croc{M_{k,2}^N}(N^{d-1}t)\right)}{N^{2k+1}}=0.
\]
Doob's Inequality shows that the martingale terms of Relation~\eqref{eqsaux4} scaled by $N^{k+1/2}$ vanish for the convergence of processes when $N$ gets large. The proposition is proved. 
\end{proof}
\begin{prop}\label{limit}
Under the condition $d\beta\mu<\lambda$ and  if $X^N(0)=(0,\ldots,0,F_N)$ then,  for the convergence in distribution of continuous processes,  the relations
\begin{equation}\label{eqaa}
\lim_{N\to+\infty} \left(\sqrt{N}\left(\frac{\mu}{N}\int_0^{N^{d-1}t}X_1^N(u)\,\diff u-\lambda \frac{(d{-}1)!}{\rho^{d-1}} \int_0^t X_{d-1}^N(N^{d-1}u)\,\diff u\right)\right)=0
\end{equation}
and
\begin{equation}\label{eqbb}
\lim_{N\to+\infty} \left(\frac{X_0^N(N^{d-1}t)}{N}-\lambda \frac{(d{-}1)!}{\rho^{d-1}} \int_0^t X_{d-1}^N(N^{d-1}u)\,\diff u\right)=0
\end{equation}
hold. 
\end{prop}
\begin{proof}
By  Relation~\eqref{id1}, one gets that, for the convergence in distribution of continuous processes,
\[
\lim_{N\to+\infty} \left( \sqrt{N}\left(\int_0^{N^{d-1}t} 
\left[\mu(k+1)  \frac{X_{k+1}^N(u)}{N^{k+1}}-\lambda \frac{X_k^N(u)}{N^k}\right]\,\diff u\right)\right)=0,
\]
holds for $1\leq k\leq d-2$, and therefore that
\[
\lim_{N\to+\infty} \left(\sqrt{N} \left(\int_0^{N^{d-1}t} 
\left[ \frac{(k{+}1)!}{\rho^{k+1}}  \frac{X_{k+1}^N(u)}{N^{k+1}}-\frac{k!}{\rho^{k}}\frac{X_k^N(u)}{N^k}\right]\,\diff u\right)\right)=0.
\]
By summing up these relations, one finally gets that
\[
\lim_{N\to+\infty} \left(\sqrt{N}\left(\frac{(d{-}1)!}{\rho^{d-1}} \frac{1}{N^{d-1}}\int_0^{N^{d-1}t}   X_{d-1}^N(u)\,\diff u 
-\frac{\mu}{\lambda}\int_0^{N^{d-1}t} \frac{X_1^N(u)}{N}\,\diff u\right)\right)=0.
\]
Relation~\eqref{eqaa} is proved. 

SDE~\eqref{SDE11} for $(X_0^N(t))$ gives the relation
\begin{align*}
\frac{X_0^N(N^{d-1}t)}{N}=\frac{M_0^N(N^{d-1}t)}{N}+\frac{\mu}{N}\int_0^{N^{d-1}t} X_1^N(u)\,\diff u,
\end{align*}
where $({M_0^N(N^{d-1}t)}/{N})$ is a martingale whose previsible increasing process is given by
\[
\left(\croc{\frac{M_0^N}{N}}(t)\right)=\left(\mu\frac{1}{N^2}\int_0^{N^{d-1}t} X_1^N(u)\,\diff u\right),
\]
it is converging in distribution to $0$ by Proposition~\ref{prop1}, one concludes that the martingale is also converging to $0$. The proposition is thus proved. 
\end{proof}
We now turn to the proof of an averaging principle.   It relies on the martingale characterization of  Markov processes  as used in Papanicolau et al.~\cite{PSV} in a Brownian setting, see also Kurtz~\cite{Kurtz:05}.
\subsection{Convergence of Occupation Measures}\label{StocAv}
For $x\in\N$ and $N\geq 1$, the random measure $\Lambda_x^N$ on $\R_+$ is defined as, for a measurable function $g:\R_+\to\R_+$,
\[
\croc{\Lambda_x^N,g}=\int_{\R_+} g(t)\ind{X^N_{d-1}(N^{d-1}t)=x}\,\diff t.
\]
Clearly $\Lambda_x^N$ is the random  Radon measure  associated with the local time of $(X_{d-1}^N(t))$ at $x$. For a given $x$,  the sequence $(\Lambda_x^N)$ of random Radon measures on $\R_+$ is tight. See Dawson~\cite[Lemma~3.28, page~44]{Dawson} for example. Note that the  null measure  can be a possible limit of this sequence. By using a diagonal argument, one can fix $(N_k)$ such that, for any $x\in\N$, $(\Lambda_x^{N_k})$ is a converging subsequence whose limit is $\nu_x$. 

Since, for $N\geq 1$, $\Lambda_x^{N}$ is absolutely continuous with respect to the Lebesgue measure on $\R_+$, the same property holds for a possible limiting measure $\nu_x$. Let $(x,t)\to \pi_t(x)$ denote its (random) density. It should be remarked that, one can choose a version of $\pi_t(x)$ such that the map $(\omega,x,t)\to \pi_t(x)(\omega)$ on the product of the probability space and $\N\times\R_+$ is measurable by taking $\pi_t(x)$ as a limit of measurable maps,
\[
\pi_t(x)=\limsup_{s\to 0}\frac{1}{s} \nu_x([t,t+s]).
\]
See Chapter~8 of Rudin~\cite{Rudin:01} for example. See also Lemma~1.4 of Kurtz~\cite{Kurtz:05}.  One denotes by $\pi_t$ the measure on $\N$ defined by  the sequence $(\pi_t(x),x\in\N)$. 
\begin{prop}\label{proppsi}
For any function $f:\N\to \R_+$ such that the sequence $(f(x)/x)$ is bounded then,  with the subsequence $(N_k)$ defined above, 
for the convergence in distribution of continuous processes,
\begin{align*}
\lim_{k\to+\infty} \left(\frac{1}{N_k^{d-1}}\int_0^{N_k^{d-1}t}f\left(X_{d-1}^{N_k}(u)\right)\,\diff u\right)
=\left( \int_0^t \croc{\pi_u,f}\,\diff u\right),
\end{align*}
In particular, almost surely, for all $t\geq 0$,
\[
\sum_{x\geq 0}\int_0^t\pi_u(x)\,\diff u=\int_0^t\pi_u(\N)\,\diff u =t.
\]
\end{prop}
\begin{proof}

Denote  $K=\sup\{f(x)/x:x\geq 1\}$ and  
\[
\Psi_f^N(t)= \frac{1}{N^{d-1}}\int_0^{N^{d-1}t}f\left(X_{d-1}^{N}(u)\right)\,\diff u,
\]
the stochastic domination results of Proposition~\ref{propdom} gives that, for any $0\leq s\leq t$, 
\[
\Psi_f^N(t)-\Psi_f^N(s)\leq K \frac{1}{N^{d}}\int_{N^{d}s}^{N^{d}t}L_0(u)\,\diff u,
\]
where $(L_0(t))$  is the process of the number of customers of  an  $M/M/1$ queue with arrival rate $d\mu\beta_0$ and service rate $\lambda$ for some convenient $\beta_0$ such that $d\mu\beta_0<\lambda$ and  with the initial condition $L_{0}(0)=0$. 
The convergence result of Lemma~\ref{lemLGN} implies then that the sequence of processes $(\Psi_f^N(t))$ is tight by the criteria of the modulus of continuity. 

For $C\geq 1$ and $t>0$
\begin{align*}
\frac{1}{N^{d-1}}\int_0^{N^{d-1}t}\!\!\!\!\!f(X_{d-1}^{N}(u))\ind{X^N_{d-1}(u)\geq C}\,\diff u
&\leq \frac{K}{N^{d-1}}\int_0^{N^{d-1}t} \!\!\!\!\!X_{d-1}^{N}(u)\ind{X^N_{d-1}(u)\geq C}\,\diff u\\
&\leq \frac{K}{N^{d}}\int_0^{N^{d}t}L_0(u)\ind{L_0(u)\geq C}\,\diff u.
\end{align*}
The last term is converging in distribution to $Kt \E(\bar{L}_0)\ind{\bar{L}_0\geq C})$, where $\bar{L}_0$ is a random variable with geometric distribution with parameter $d\mu\beta_0/\lambda$, the invariant distribution of the process $(L_0(t))$.  In particular for $T>0$, if $C$ is sufficiently large, this term can be made arbitrarily small uniformly for $t\leq T$.

By using the fact that, for $x\in \N$,
\[
\frac{1}{N^{d-1}}\int_0^{N^{d-1}t}f(X_{d-1}^{N}(u))\ind{X^N_{d-1}(u)=x}\,\diff u= 
f(x)\croc{\Lambda_x^N,\mathbbm{1}_{[0,t]}},
\]
one gets the desired  convergence in distribution.

\end{proof}

\subsection{The Decay of the Network occurs on the Time Scale $t\mapsto N^{d-1}t$}

We have all the necessary technical results to prove the main result concerning the behavior of the system on the time scale $t\mapsto N^{d-1} t$. 
\begin{theorem}[Rate of Decay of the Network]\label{theodec}
Under the condition $d\beta\mu<\lambda$ and  if $X^N(0)=(0,\ldots,0,F_N)$,
then  the sequence of  processes $(X_0^N(N^{d-1}t)/N)$
converges in distribution to $(\Phi(t))$ where, for $t\geq 0$,  $\Phi(t)$ is the unique solution
$y\in[0,\beta]$ of the equation
\begin{equation}\label{decay}
\left(1-\frac{y}{\beta}\right)^{\rho/d} e^y=\exp\left(-\lambda\frac{(d{-}1)!}{\rho^{d-1}} t\right).
\end{equation}
\end{theorem}
\begin{proof}
Let $f$ be a function on $\N$ with finite support then, the SDE~\eqref{SDE1k} associated to the evolution equations give
\begin{align*}
&\frac{f(X_{d{-}1}^N(N^{d{-}1}t))-f(0)-M_{f}^{N}(N^{d{-}1}t)}{N^d}=\\
&\int_0^t\Delta^+(f)(X_{d{-}1}^N(N^{d{-}1}u)) \left(d\mu\frac{X_d^N(N^{d{-}1}u)}{N}+\lambda\ind{S_{d-3}^N=0,X_{d-2}^N(N^{d{-}1}u)>0}\right)\,\diff u\\
&{+}\!\!\int_0^t \Delta^-(f)(X_{d{-}1}^N(N^{d{-}1}u))
\!\!\left(\!\!(d{-}1)\mu \frac{X_{d{-}1}^N(N^{d{-}1}u)}{N}{+}\lambda \ind{S_{d-2}^N{=}0,X_{d{-}1}^N(N^{d{-}1}u)>0}\!\!\right)\,\diff u,
\end{align*}
where $\Delta^+(f)(x)=f(x+1)-f(x)$ and $\Delta^-(f)(x)=f(x-1)-f(x)$. 
The convergence of the various components of this identity are now examined. 

Clearly enough, $f$ being bounded, the  process $([f(X_{d{-}1}^N(N^{d{-}1}t))-f(0)]/N^d)$ is converging in distribution to $0$ as $N$ gets large. By calculating  the previsible increasing process of the martingale $(M_{f}^{N}(N^{d{-}1}t)/N^d)$, it is not difficult to show that this process vanishes at infinity. 

Note that
\begin{multline*}
\left|\int_0^t\Delta^+(f)(X_{d{-}1}^N(N^{d{-}1}u)) \ind{S_{d-3}^N=0,X_{d-2}^N(N^{d{-}1}u)>0}\,\diff u\right|\\\leq 
2\|f\|_{\infty} \int_0^t X_{d-2}^N(N^{d{-}1}u)\,\diff u=
2\|f\|_{\infty} \frac{1}{N^{d{-}1}}\int_0^{N^{d{-}1}t} X_{d-2}^N(u)\,\diff u,
\end{multline*}
 the process associated to the last term is converging in distribution to $0$ by Proposition~\ref{prop1}.
Similarly,
\begin{multline*}
\left|\int_0^t \Delta^-(f)(X_{d{-}1}^N(N^{d{-}1}u)) \ind{S_{d-2}^N>0,X_{d{-}1}^N(N^{d{-}1}u)>0}\,\diff u\right|
\\\leq 2\|f\|_\infty \sum_{k=1}^{d-2}\int_0^t X_{k}(N^{d{-}1}u)\,\diff u 
\end{multline*}
and the last term is also converging to $0$ in distribution. 
In the same way
\[
\left|\int_0^t \Delta^-(f)(X_{d{-}1}^N(N^{d{-}1}u)) \frac{X_{d{-}1}^N(N^{d{-}1}u)}{N}\,\diff u\right|
\leq 
2\|f\|_{\infty} \int_0^t \frac{X_{d{-}1}^N(N^{d{-}1}u)}{N}\,\diff u
\]
which converges to $0$ by the last assertion of Proposition~\ref{propdom}.

To summarize, we have proved that the following convergence  in distribution
\begin{multline}\label{eqaux}
\lim_{N\to+\infty} \left(\int_0^t\Delta^+(f)(X_{d-1}^N(N^{d-1}u)) d\mu\frac{X_d^N(N^{d-1}u)}{N}\,\diff u \right.\\\left.
+\int_0^t\Delta^-(f)(X_{d-1}^N(N^{d-1}u))\lambda
\ind{X_{d-1}^N(N^{d-1}u)>0}\,\diff u \right)=0.
\end{multline}

By using again Propositions~\ref{limit} and~\ref{proppsi}, one gets that the convergence of the sequence of processes
is converging to a continuous process $(\Phi(t))$ such that
\[
\left(\Phi(t)\right)\stackrel{\text{def.}}{=}\lim_{N\to+\infty} \left(\frac{X_0^{N_k}(N_k^{d-1}t)}{N_k}\right) =\left( \lambda \frac{(d{-}1)!}{\rho^{d-1}} \int_0^t \croc{\pi_u,I}\,\diff u\right),
\]
where $I(x)=x$ for $x\geq 0$. 
 By  the Skorohod representation theorem, one can  take a convenient probability space such  that, for all $x\in\N$,   this convergence also holds almost surely as well as the convergence of the processes   $(\croc{\Lambda_x^{N_k},\mathbbm{1}_{[0,t]}})$.  The identity
\[
\frac{X_d^N(N^{d-1}t)}{N}=\frac{F_N}{N}-\frac{X_0(N^{d-1}t)}{N}-\sum_{k=1}^{d-1} \frac{X_k^N(N^{d-1}t)}{N},
\]
and  Equation~\eqref{eqaux} give that  the relation
\[
\int_0^t\sum_{x\in\N} \pi_u(x)\left[ d\mu(\beta-\Phi(u))\Delta^+(f)(x)
+\lambda
\ind{x>0}\Delta^-(f)(x)\rule{0mm}{5mm}\right]\,\diff u =0
\]
holds almost surely for all $t\geq 0$ and all functions $f=f_k$, $k\geq 0$ with $f_k(x)=\mathbbm{1}_k(x)$ for $x\in\N$.  One concludes from  this relation and Proposition~\ref{proppsi}, for $u\in\R_+$ outside  a set ${\cal S}$ negligible for the  Lebesgue measure, one has for all $k\geq 0$
\[
\sum_{x\in\N} \pi_u(x)\left[ d\mu(\beta-\Phi(u))\Delta^+(f_k)(x)+\lambda \ind{x>0}\Delta^-(f_k)(x)\rule{0mm}{5mm}\right]=0
\]
and $\pi_u(\N)=1$.  Hence, if $u\in {\cal S}$,  $(\pi_u(x))$ is a geometric distribution,  the invariant distribution of an $M/M/1$ queue with arrival rate
$d\mu(\beta -\Phi(u))$
and service rate $\lambda$. The definition of $\Phi(t)$ gives therefore the fixed point equation, for all $t\geq 0$, 
\begin{equation}\label{intPhi}
\Phi(t)= \lambda \frac{(d{-}1)!}{\rho^{d-1}} \int_0^t\frac{d\mu\left(\beta -\Phi(u)\right)}{\lambda-d\mu\left(\beta -\Phi(u)\right)}  \,\diff u,
\end{equation}
one gets the relation
\[
\left(1-  \frac{\Phi(t)}{\beta}\right)^{\rho/d}e^{\Phi(t)}=\exp\left(- \lambda\frac{(d{-}1)!}{\rho^{d-1}}t\right).
\]
The theorem is proved. 
\end{proof}

One concludes this section with the asymptotic of the first instant when the network has lost a fraction $\delta\in(0,1)$ of its file. It generalizes Corollary~1 of Feuillet and Robert~\cite{Feuillet-Robert}.  This is a direct consequence of the above theorem. 
\begin{corollary}
If, for $\delta\in(0,1)$,
\[
T_N(\delta)=\inf\left\{t\geq 0: \frac{X_0^N(t)}{N}\geq \delta \beta\right\}
\]
then, under the condition $\lambda >d\beta\mu$,  the relation
\[
\lim_{N\to+\infty} \frac{T_N(\delta)}{N^{d-1}}=\frac{\rho^{d-1}}{\lambda(d-1)!}\left(-\frac{\rho}{d}\log(1-\delta)-\beta\delta\right)
\]
holds for the convergence in distribution.
\end{corollary}

\section{Second Order Asymptotics in the Stable Case}\label{CLT}
This section is devoted to the study of the second order fluctuations associated to the law of large numbers proved in Theorem~\ref{theodec}. As it will be seen the  proof relies on  careful stochastic calculus, technical estimates and Proposition~\ref{lmtitg} proved in Section~\ref{AverageSec}. 

\medskip
\noindent
{\bf Notations}
\begin{enumerate}
\item[---] If $(Y_N(t))$ and $(Z_N(t))$ are sequences of stochastic process, with a slight abuse of notation, we will write $Z_N(t)=Y_N(t)+{\cal O}_d(1)$ when the sequence $(Z_N(t)-Y_N(t))$ converges in distribution to $0$ when $N$ goes to infinity. 
\end{enumerate}

\begin{lemma}\label{lemserv}
Let 
\[
Y_{d-1}^N(t){\stackrel{\text{def.}}{=}}\frac{1}{\sqrt{N}}\int_0^{t}\left[(X_{d-1}^N(N^{d-1}u)+1)d\mu X_d^N(N^{d-1}u)-\lambda N X_{d-1}^N(N^{d-1}u)\right]\diff u\\
\]
then $(Y_{d-1}^N(t))$ converges in distribution  to $0$ as $N$ goes to infinity. 
\end{lemma}
\begin{proof}
By using the SDE satisfied by the process $(X_{d-1}^N(t))$, as in the proof in Proposition \ref{lmtitg}, one gets
\begin{align}
\frac{X_{d-1}^N(N^{d-1}t)}{N^{d-1/2}}&+\frac{X_{d-1}^N(N^{d-1}t)^2}{N^{d-1/2}}=\frac{M_{d-1}^N(t)}{N^{d-1/2}}+\frac{M_{d-1,2}^N(t)}{N^{d-1/2}}\notag\\
&+\frac{2}{N^{d-1/2}}\int_0^{N^{d-1}t}\left[(X_{d-1}^N(u)+1)\lambda N\ind{S_{d-3}^N(u)=0,X^N_{d-2}(u)>0}\right.\notag\\ &\hspace{2cm}\left.-X_{d-1}^N(u)((d-1)X_{d-1}^N(u)-\lambda N\ind{S_{d-2}^N(u)>0})\right]\diff u\notag\\
&+\frac{2}{N^{d-1/2}}\int_0^{N^{d-1}t}[(X_{d-1}^N(u)+1)d\mu X_d^N(u)-\lambda N X_{d-1}^N(u)]\diff u\notag
\end{align}
where $(M_{d-1}^N(t))$ and $(M_{d-1,2}^N(t))$ are the associated local martingales.
The processes of left hand side of this relation vanishes as $N$ gets large due to Proposition~\ref{propdom}. With similar arguments as in the proof of Proposition~\ref{lmtitg}, one obtains that the martingale terms and the first integral of the right hand side vanish too. This is again a consequence of Propositions~\ref{propdom} and~\ref{prop1}. 

Therefore,  the last term  
\[
\left(\frac{1}{N^{d-1/2}}\int_0^{N^{d-1}t}\left[(X_{d-1}^N(u)+1)d\mu X_d^N(u)-\lambda N X_{d-1}^N(u)\right]\diff u\right)
\]
 converges to $0$ in distribution when $N$ gets large.  The lemma is proved. 
\end{proof}

\begin{theorem}[Central Limit Theorem]\label{theoCLT}
If  $d\beta\mu{<}\lambda$ and   if Condition~\eqref{beta} holds and the initial state is $X^N(0)=(0,\ldots,0,F_N)$,
then  the following  convergence in distribution holds
\[
\lim_{N\to+\infty} \left(\frac{X_0^N(N^{d-1} t)-N\Phi(t)}{\sqrt{N}}\right)=(W(t)),
\]
where $\Phi(t)$ is the solution of Equation~\eqref{decay} and  the process $(W(t))$ is the solution of  the stochastic differential equation
\begin{equation}\label{SDECLT}
\diff W(t)=\sqrt{\Phi'(t)}\diff B(t)-
\frac{\lambda^2 \mu d!}{\rho^{d-1}}
\frac{W(t)-\gamma}{(\lambda-d\mu(\beta-\Phi(t)))^2}\diff t,
\end{equation}
with $W_0(0)=0$, where $({B}(t))$ is a standard Brownian motion and $\Phi(t)$ is the unique solution of  Equation~\eqref{decay}. 
\end{theorem}
\begin{proof}
We denote by
\[
W_0^N(t)=\frac{X_0^N(N^{d-1} t)-N\Phi(t)}{\sqrt{N}} \text{ and }
W_d^N(t)=\frac{X_d^N(N^{d-1} t)-N(\beta-\Phi(t))}{\sqrt{N}}.
\]
The strategy of the proof consists in starting from the convergence proved in the above lemma to write an integral equation for the process $(W_0^N(t))$, this is Equation~\eqref{eqdd} below. Technical results of Section~\ref{AverageSec} are used repeatedly in the proof of this identity. The last part of the proof  consists in  proving the tightness and identifying the possible limits of this sequence. 

The total sum of the coordinates of  $(X_k^N(t))$ being $F_N$, scaling Condition~\eqref{beta} and Relation~\eqref{serv13} of Proposition~\ref{propdom} give the identity
\begin{multline}\label{eqcc}
W_d^N(t)=\frac{X_d^N(N^{d-1}u)-N(\beta-\Phi(t))}{\sqrt{N}}=\frac{F_N-N\beta}{\sqrt{N}}-\sum_{k=1}^{d-1}\frac{X_k^N(N^{d-1}u)}{\sqrt{N}}\\
-\frac{X_0^N(N^{d-1}u)-N\Phi(t)}{\sqrt{N}}=-W_0^N(t)+\gamma+{\cal O}_d(1).
\end{multline}
The SDE~\eqref{SDE11} gives the relation
\[
X_0^N(N^{d-1}t)=\mu\int_0^t X_1^N(N^{d-1}u)N^{d-1}\,\diff u +M_0^N(N^{d-1}t). 
\]
The previsible increasing process of the martingale $(M_0^N(N^{d-1}t)/\sqrt{N})$ is given by
\[
\left(\croc{\frac{M_0^N}{\sqrt{N}}}(N^{d-1}t)\right)=\left(\mu\frac{1}{N}\int_0^{N^{d-1}t} X_1^N(u)\,\diff u\right),
\]
and it is converging in distribution  to $(\Phi(t))$, see the proof of Proposition~\ref{limit}. Consequently, by using  Theorem~1.4 page~339 of  Ethier and Kurtz~\cite{Ethier}  for example, for the convergence in distribution of processes, one has 
\[
\lim_{N\to+\infty} \left(\frac{M_0^N}{\sqrt{N}}\right)=\left(\int_0^t \sqrt{\Phi'(u)}\,\diff B(u)\right)\stackrel{\text{dist.}}{=} (B(\Phi(t))),
\] 
where $(B(t))$ is a standard Brownian motion on $\R$.

Let 
\[
H_N(t)=\int_0^{t}X_{d-1}^N(N^{d-1}u)\,\diff u,
\]
Relation~\eqref{eqaa} of Proposition~\ref{limit} shows that
\[
\frac{\lambda(d{-}1)!}{\rho^{d-1}} \sqrt{N}H_N(t) =\sqrt{N}\mu\int_0^{N^{d-1}t} \frac{X_1^N(u)}{N}\,\diff u+{\cal O}_d(1)
\]
holds and   SDE~\eqref{SDE11} gives
\begin{multline*}
\mu\int_0^{N^{d-1}t} X_1^N(u)\,\diff u=X_0^N(N^{d-1}t)-M_0^N(N^{d-1}t)
\\=N\Phi(t)+\sqrt{N} W_0^N(t)-M_0^N(N^{d-1}t).
\end{multline*}
One obtains therefore the following expansion for $(\sqrt{N}H_N(t))$, 
\begin{equation}\label{eqH}
\sqrt{N}\left(\frac{\lambda(d{-}1)!}{\rho^{d-1}} H_N(t) - \Phi(t)\right)=W_0^N(t)-\frac{M_0^N(N^{d-1}t)}{\sqrt{N}}+{\cal O}_d(1).
\end{equation}
Lemma~\ref{lemserv} gives the relation
\[
\frac{1}{\sqrt{N}}\int_0^{t}\left[(X_{d-1}^N(N^{d-1}u)+1)d\mu X_d^N(N^{d-1}u)-\lambda N X_{d-1}^N(N^{d-1}u)\right]\diff u={\cal O}_d(1),
\]
which can be rewritten as 
\begin{multline*}
d\mu{\sqrt{N}}\int_0^{t} X_{d-1}^N(N^{d-1}u) \left(\frac{X_d^N(N^{d-1}u)}{N}-(\beta-\Phi(u))) \right)\diff u\\
+{\sqrt{N}}\int_0^{t} X_{d-1}^N(N^{d-1}u)\left( d\mu (\beta-\Phi(u))-\lambda \right)\diff u
\\+d\mu\sqrt{N}\int_0^{t}(\beta-\Phi(u))\diff u
+d\mu\frac{1}{\sqrt{N}}\int_0^{t}(X_d^N(N^{d-1}u)-N(\beta-\Phi(u))\diff u ={\cal O}_d(1).
\end{multline*}
If one plugs the integration by part 
\begin{multline*}
\int_0^{t}X_{d-1}^N(N^{d-1}u)\left[\lambda -d\mu (\beta-\Phi(u))\right] \diff u
\\=H_N(t)\left[\lambda -d\mu (\beta-\Phi(t))\right]- d\mu \int_0^{t}H_N(u)\Phi'(u) \diff u,
\end{multline*}
into this identity, this gives the relation
\begin{multline*}
d\mu\int_0^{t} X_{d-1}^N(N^{d-1}u) W_d^N(u)\diff u\\
-{\sqrt{N}}H_N(t)\left[\lambda -d\mu (\beta-\Phi(t))\right]+ d\mu \int_0^{t}\sqrt{N}H_N(u)\Phi'(u) \diff u
\\+d\mu\sqrt{N}\int_0^{t}(\beta-\Phi(u))\diff u
+d\mu\int_0^{t}W_d^N(u)\diff u ={\cal O}_d(1). 
\end{multline*}
The expansion~\eqref{eqH} for $(\sqrt{N}H_N(t))$ yields
\begin{multline*}
d\mu\int_0^{t} X_{d-1}^N(N^{d-1}u) W_d^N(u)\diff u+\sqrt{N}\Delta_N(t)\\
-\frac{\rho^{d-1}}{\lambda(d{-}1)!}\left(W_0^N(t)-\frac{M_0^N(N^{d-1}t)}{\sqrt{N}}\right)\left[\lambda -d\mu (\beta-\Phi(t))\right]\\+ d\mu \frac{\rho^{d-1}}{\lambda(d{-}1)!}\int_0^{t}\left(W_0^N(u)-\frac{M_0^N(N^{d-1}u)}{\sqrt{N}}\right)\Phi'(u) \diff u\\
+d\mu\int_0^{t}W_d^N(u)\diff u ={\cal O}_d(1)
\end{multline*}
with
\begin{align*}
\Delta_N(t)&=\frac{\rho^{d-1}}{\lambda(d{-}1)!}\left(-\Phi(t)\left[\lambda -d\mu (\beta-\Phi(t))\right]+ d\mu \int_0^{t}\Phi(u)\Phi'(u) \diff u\right)
\\& \qquad+d\mu\int_0^{t}(\beta-\Phi(u))\diff u\\
&=-\frac{\rho^{d-1}}{\lambda(d{-}1)!} \int_0^t \left[\lambda -d\mu (\beta-\Phi(u))\right]\Phi'(u)\,\diff u+d\mu\int_0^{t}(\beta-\Phi(u))\diff u\\
&=0,
\end{align*}
by Relation~\eqref{intPhi}. 
By using Equation~\eqref{eqcc}, one gets finally the relation
\begin{multline}\label{eqdd}
-d\mu\int_0^{t} X_{d-1}^N(N^{d-1}u) W_0^N(u)\diff u+d\mu\gamma\int_0^t X_{d-1}^N(N^{d-1}u)\,\diff u\\
-\frac{\rho^{d-1}}{\lambda(d{-}1)!}\left(W_0^N(t)-\frac{M_0^N(N^{d-1}t)}{\sqrt{N}}\right)\left[\lambda -d\mu (\beta-\Phi(t))\right]\\+ d\mu \frac{\rho^{d-1}}{\lambda(d{-}1)!}\int_0^{t}\left(W_0^N(u)-\frac{M_0^N(N^{d-1}u)}{\sqrt{N}}\right)\Phi'(u) \diff u\\-d\mu\int_0^{t}W_0^N(u)\diff u +d\mu\gamma t={\cal O}_d(1).
\end{multline}
Starting from the above equation, one can now complete the proof of the theorem in four steps. 
\begin{enumerate}
\item Local boundedness.\\
By using the convergence in distribution of $({M_0^N(N^{d-1}u)}/{\sqrt{N}})$ and Gronwall's Inequality, one gets that, 
for $\eps>0$ and  $T>0$, there exists some $K>0$ and $N_0$ such that if $N\geq N_0$, then 
\begin{equation}\label{eqauxAK}
\P\left(\sup_{0\leq s\leq T} |W_0^N(s)| \geq K\right)\leq \eps.
\end{equation}
\item Tightness.\\
One first note that  the two sequences of processes
\[
\left(\int_0^{t} X_{d-1}^N(N^{d-1}u) \diff u\right) \text{ and }\left(\int_0^{t} X_{d-1}^N(N^{d-1}u) W_0^N(u)\diff u\right) 
\]
satisfy the criterion of the modulus of continuity: for the first sequence this is a consequence of Proposition~\ref{limit} and Theorem~\ref{theodec}. 
Relation~\eqref{eqauxAK} and the fact that, for $0\leq s \leq t\leq T$,
\[
\left|\int_s^{t} X_{d-1}^N(N^{d-1}u) W_0^N(u)\diff u\right|\leq\left(\sup_{0\leq u\leq T} |W_0^N(u)| \right) \int_s^t X_{d-1}^N(N^{d-1}u) \diff u,
\]
give this property for the second sequence. As it has been seen this is also the case for $({M_0^N(N^{d-1}u)}/{\sqrt{N}})$. 
Relation~\eqref{eqdd} can thus be rewritten as
\begin{equation}\label{eqKA}
W_0^N(t)+\int_0^{t}W_0^N(u)F(u) \diff u=H^N(t),
\end{equation}
where $(F(t))$ is a deterministic continuous function and $(H^N(t))$ is a sequence of processes which satisfies the criterion of the modulus of continuity. 
As before, See relation~\eqref{modul},  denote $w_{Z}$ as the modulus  of continuity of the process $(Z(t))$ on $[0,T]$, Relation~\eqref{eqKA} gives the inequality,
\[
w_{W_0^N} \leq w_{H^N}+\delta \|F\|_\infty  \sup_{0\leq s\leq T} |W_0^N(s)|  
\]
with $\|F\|_\infty=\sup(|F(s)|,0{\leq}s{\leq}T)$. One deduces the tightness of $(W_0^N(t))$  by the criterion of the modulus of continuity. In particular any limiting point is a continuous process. 
\item Convergence of the first term of Equation~\eqref{eqdd}.\\
Let $(W(t))$ be a limit of some subsequence $(W_0^{N_k}(t))$. By Skorohod's representation theorem, on can assume that the
convergence
\[
\lim_{k\to+\infty} \left(\int_0^{t} X_{d-1}^{N_k}(N_k^{d-1}u) \diff u, W_0^{N_k}(t)\right)=\left(\frac{\rho^{d-1}}{\lambda(d{-}1)!}\Phi(t),W(t)\right)
\]
holds almost surely for the uniform norm on compact sets of $\R_+$. If $f$ is a $C^1(\R_+)$ function, by integration par parts, one has the convergence
\[
\lim_{k\to+\infty} \left(\int_0^{t} X_{d-1}^{N_k}(N_k^{d-1}u) f(u)\diff u\right)=\left(\frac{\rho^{d-1}}{\lambda(d{-}1)!}\int_0^t\Phi'(u)f(u)\,\diff u\right),
\]
which can be extended to any arbitrary continuous function $f$ by a regularization procedure. Since
\[
\lim_{k\to+\infty} \left(\int_0^{t} X_{d-1}^{N_k}(N_k^{d-1}u) (W_0^{N_k}(u)-W(u))\,\diff u\right)=0,
\]
one finally gets the convergence
\[
\lim_{k\to+\infty} \left(\int_0^{t} X_{d-1}^{N_k}(N_k^{d-1}u) W_0^{N_k}(u)\,\diff u\right)=\left(\frac{\rho^{d-1}}{\lambda(d{-}1)!}\int_0^t\Phi'(u)W(u)\,\diff u\right),
\]
\item Identification of the limit. \\
A possible limit $(W(t))$ satisfies therefore the integral equation
\begin{multline*}
-d\mu\int_0^{t} \left(\frac{\rho^{d-1}}{\lambda(d{-}1)!}\Phi'(u)+1\right) W(u)\diff u+ d\gamma\frac{\rho^{d-2}}{(d{-}1)!}\Phi(t)\\
-\frac{\rho^{d-1}}{\lambda(d{-}1)!}\left(W(t)-B(\Phi(t))\right)\left[\lambda -d\mu (\beta-\Phi(t))\right]\\+ d\mu \frac{\rho^{d-1}}{\lambda(d{-}1)!}\int_0^{t}\left(W(u)-B(\Phi(u))\right)\Phi'(u) \diff u+d\mu\gamma t=0,
\end{multline*}
and, with Relation~\eqref{intPhi}, it can be rewritten as
\begin{multline*}
-\lambda d\mu\int_0^{t} \frac{W(u)}{\lambda-(\beta-\Phi(u))}\,\diff u+ d\gamma\frac{\rho^{d-2}}{(d{-}1)!}\Phi(t)+d\mu\gamma t\\
-\frac{\rho^{d-1}}{\lambda(d{-}1)!}\int_0^t \left[\lambda -d\mu (\beta-\Phi(u))\right] \left(\diff W(u)-\sqrt{\Phi'(u)}\diff B(u)\right)=0.
\end{multline*}
The theorem is proved. 
\end{enumerate}

\end{proof}
\section{A Local Equilibrium in the Overloaded Case}\label{localsec}
\begin{figure}
        \centering
\scalebox{0.4}{\includegraphics{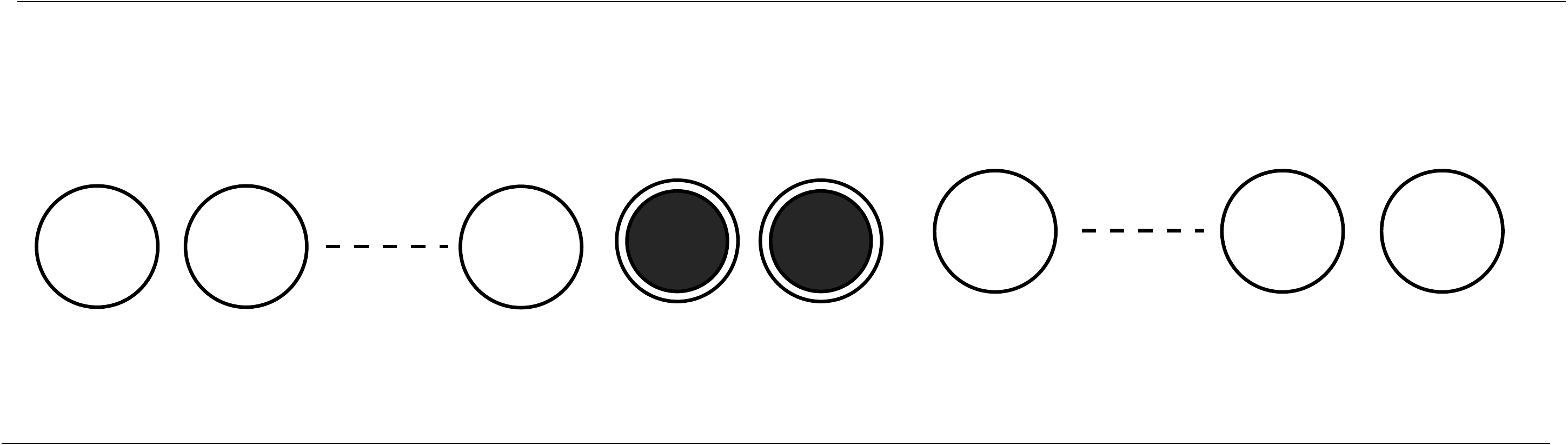}}
\put(-320,20){$X_0$}
\put(-285,20){$X_1$}
\put(-230,20){$X_{p{-}1}$}
\put(-195,20){$X_p$}
\put(-210,75){$\displaystyle \scriptscriptstyle{\left[ (p{+}1)\beta{-}\frac{\lambda}{\mu}\right]N}$}
\put(-165,20){$X_{p{+}1}$}
\put(-160,75){$\displaystyle \scriptscriptstyle{\left[ \frac{\lambda}{\mu}{-}p\beta\right]N}$}
\put(-130,20){$X_{p{+}2}$}
\put(-70,20){$X_{d-1}$}
\put(-30,20){$X_{d}$}
 \caption{Stable Asymptotic Fluid State in an Overloaded Network  with $p\beta\mu<\lambda<(p+1)\beta\mu$ for some $1< p < d$}  \label{fig1}
\end{figure}
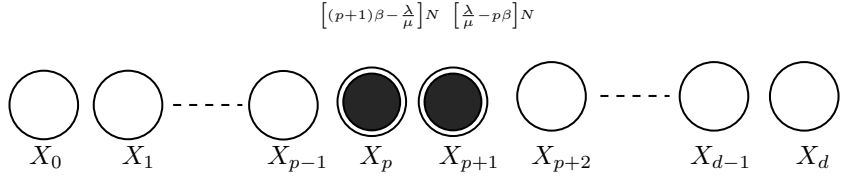

We have seen in Corollary~\ref{fluid2} that 
if for some $2\leq p<d$, one has $p\beta\leq \rho<(p+1)\beta$ and if the initial state is $X_N(0)=(0,\ldots,0,F_N)$ then one has the convergence in distribution
\[
\lim_{N\to+\infty} \frac{1}{N}(X_p^N(t),X_{p+1}^N(t))=(x_p(t),x_{p+1}(t))
\]
and 
\[
\lim_{t\to+\infty}(x_p(t),x_{p+1}(t))= ((p+1)\beta -\rho, \rho - p \beta). 
\]
The system started with ${\sim}\beta N$ files with $d$ copies and it ends up, on the normal time scale, in a state where there are still $~\beta N$ files but with either $p$ or $p+1$ copies.  

In this section we  start from this ``equilibrium'', Proposition~\ref{propstay}  shows that this fluid state  does not change  on  the time scale $t\mapsto N^{p-2}t$.  Theorem~\ref{theoloc} proves that,  on  the time scale $t\mapsto N^{p-1}t$,  a positive fraction of files are lost. It is also shown that the number of files with $p$ copies decreases to end up in a state where, for the fluid state, there are only files with $p+1$ copies. 

One starts with an elementary result concerning the $M/M/\infty$ queue. 
\begin{lemma}
If $(L_N(t))$ is the Markov process associated to an $M/M/\infty$ queue with arrival rate $\lambda N$ and service rate $\mu$, and initial condition such that 
\[
\lim_{N\to+\infty} \frac{L_N(0)}{N}=\frac{\lambda}{\mu},
\]
then, for any $\ell\in\N$, the convergence in distribution
\[
\lim_{N\to+\infty} \left(\frac{L_N(N^\ell t)}{N}\right)=\frac{\lambda}{\mu}
\]
holds. 
\end{lemma}
\begin{proof}
For $\eps>0$,  by bounding the  rate of jumps $-1$ of the process,  a coupling can be constructed such that 
\[
L_N(t)\leq (\rho+\eps)N+\bar{L}_N(Nt)
\]
 holds for all $t\geq 0$, where $(\bar{L}_N(t))$ is an $M/M/1$ queue with input rate $\lambda$ and service rate $\lambda+\mu\eps$, with initial condition $\bar{L}_N(0)=0$. 
If $\tau_N=\inf\{t\geq 0: \bar{L}_N(t)\geq \eps N\}$ then,  Proposition~5.11 page~119 of Robert~\cite{Robert}, gives that for any $\ell\geq 1$ and $x>0$, 
\[
\lim_{N\to+\infty}\P(\tau_N\leq N^\ell x)=0.
\]
This proves that, for any $T>0$,
\[
\lim_{N\to+\infty} \P\left(\sup_{0\leq t\leq T}\frac{L_N(N^\ell t)}{N}\leq \rho+2\eps\right)=1.
\] 
With a similar argument for a lower bound one gets finally 
the convergence in distribution, for any $\ell\geq 0$,
\[
\lim_{N\to+\infty} \frac{L_N(N^\ell t)}{N}=\rho.
\]
The lemma is proved. 
\end{proof}
One shows in the next proposition that the fluid state of the network does not change on the time scale $t\mapsto N^{p-2}t$. 
\begin{prop}[Stability of Local Equilibrium on the time scale $t\mapsto N^{p-2}t$]\label{propstay}
If for some $2\leq p<d$, one has $p\beta<\rho<(p+1)\beta$, and the initial state $X_N(0)$ is such that $X_i^N(0)=0$ for $1\leq i\leq d$, $i\not\in\{p,p+1\}$ and
\[
\lim_{N\to+\infty}\left( \frac{X_p^N(0)}{N}, \frac{X_{p+1}^N(0)}{N}\right)=\left( (p+1)\beta -\rho,  \rho-p\beta\right)
\]
then for any $q\leq p-2$, for the convergence in distribution, 
\[
\lim_{N\to+\infty} \left(\frac{X_p^N(N^{q}t)}{N},\frac{X_{p+1}^N(N^{q}t)}{N}\right)=\left( (p+1)\beta -\rho,  \rho-p\beta\right)
\]
\end{prop}
\begin{proof}
Clearly it is enough to show the proposition for $q=p-2$.  Let 
\[
Z^N(t)=\sum_{k=1}^{p-1} (p-k)X_k^N(t),
\]
then, if $Z^N(t)=z$, there is a jump of size $+1$ for $Z^N$ at rate 
\[
\mu\sum_{k=2}^p kX_k^N(t)\leq p\mu\beta N,
\]
and of size $-1$ at rate $\lambda N$ if $z>0$.  In the same way as in in the proof of Proposition~\ref{propdom}, one can construct a coupling, for which 
\[
Z^N(t)\leq L_0(Nt),
\]
where $(L_0(t))$ is a stable $M/M/1$ queue with input rate $p\mu\beta$ and output rate $\lambda$.
 In particular, the convergence in distribution 
\begin{equation}\label{eqr1}
\lim_{N\to+\infty} \left(\frac{X_i^N(N^{p-2}t)}{N}\right)=0, \quad  1\leq i\leq p-1,
\end{equation}
holds.  

Because of Relation $p\mu\beta <\lambda$, one can extend the results of Propositions~\ref{prop1} and~\ref{limit} to get that, for $1\leq  k\leq p-2$, 
\[
\lim_{N\to\infty}\left( \int_0^{N^{p-2}t} 
\left[\frac{(k+1)!}{\rho^{k+1}} \frac{X_{k+1}^N(u)}{N^{k+1}}- \frac{k!}{\rho^{k}}  \frac{X_{k}^N(u)}{N^{k}}\right]\,\diff u\right)=0
\]
holds for the convergence in distribution. By summing up all these relations for $1\leq k\leq p-2$, one gets
\[
\lim_{N\to+\infty} \left(\frac{(p-1)!}{\rho^{p-1}} \frac{1}{N^{p-1}}\int_0^{N^{p-2}t}   X_{p-1}^N(u)\,\diff u 
-\frac{\mu}{\lambda}\int_0^{N^{p-2}t} \frac{X_1^N(u)}{N}\,\diff u\right)=0.
\]
Relation~\eqref{eqr1} gives the convergence in distribution
\[
\lim_{N\to+\infty} \left(\int_0^{N^{p-2}t}   \frac{X_{1}^N(u)}{N}\,\diff u \right)=0,
\]
consequently
\[
\lim_{N\to+\infty} \left(\frac{X_0^N(N^{p-2}t)}{N}\right)=0,
\]
by using the SDE associated to $(X_0^N (t))$ as in the proof of Proposition~\ref{limit}. 

One concludes that
\begin{equation}\label{j2013}
\lim_{N\to+\infty} \left(\frac{1}{N}\sum_{k=p}^{d} X_k^N(N^{p-2}t)\right)=\beta
\end{equation}
Let 
\[
Y^N(t)=\sum_{k=1}^{d} kX_k^N(t),
\]
then, if $Y^N(t)=y$, there is a jump of size $-1$ for $Y^N$ at rate  $\mu y$,
and of size $+1$ at rate $\lambda N$ if $X_1^N(t)+\cdots+(d-1)X_{d-1}^N(t)>0$. Hence, in the same way as in the proof of Proposition~\ref{propdom}, a coupling can be constructed such that the process $(Y^N(t))$ is dominated by  the process  $(L_N(t))$ of  the number of customers in an $M/M/\infty$ queue with arrival rate $\lambda N$ and service rate $\lambda$, and with initial condition such that 
\[
\lim_{N\to+\infty} \frac{L_N(0)}{N}=p((p+1)\beta-\rho)+(p+1)(\rho-p\beta)=\rho.
\]
By using the relation 
\[
\sum_{k=1}^{p-1} kX_k^N(t)+ pX_p^N(t) +(p+1)\left(\sum_{k=p}^d X_k^N(t)-X_p^N(t)\right)\leq Y^N(t)\leq L_N(t),
\]
Equations~\eqref{eqr1}~\eqref{j2013}
and the  above lemma, one gets that, for any $\eps>0$ and $T>0$,
\[
\lim_{N\to+\infty}\P\left(\inf_{0\leq t\leq T} \frac{X_p^N(N^{p-2} t)}{N}\geq (p+1)\beta-\rho-\eps\right)=1.
\]
Relation $\lambda<(p+1)\beta\mu$, gives that $(X_p^N(N^{p-2} t))$ is strictly positive on any finite interval with high probability.
Consequently,
\[
\lim_{N\to+\infty}\P\left(\inf_{0\leq t\leq T} X_1^N(N^{p-2}t)+\cdots+(d-1)X_{d-1}^N(N^{p-2}t)>1\right)=1
\]
this implies that the two processes $(Y^N(N^{p-2}t))$ and $(L_N(N^{p-2}t))$ are identical with probability close to $1$ when $N$ is large.  Secondly, since the duplication capacity cannot be used at any node with index greater than $p+1$, for any $p+2\leq k\leq d$, for the convergence in distribution, the relation
\[
\lim_{N\to+\infty}  \left(\frac{X_k^N(N^{p-2}t)}{N}\right)=0
\]
holds.  One deduces therefore the convergence in distribution
\begin{align*}
&\lim_{N\to+\infty}  \left(\frac{pX_p^N(N^{p-2}t)+(p+1)X_{p+1}^N(N^{p-2}t)}{N}\right)=\rho \\
&\lim_{N\to+\infty}  \left(\frac{X_p^N(N^{p-2}t)+X_{p+1}^N(N^{p-2}t)}{N}\right)=\beta. 
\end{align*}
The proposition is proved. 
\end{proof}
We can now state the main result of this section. 
\begin{theorem} [Evolution of the Local Equilibrium]\label{theoloc}
If for some $2\leq p<d$, one has $p\beta< \rho<(p+1)\beta$, and the initial state $X_N(0)$ is such that $X_i^N(0)=0$ for $1\leq i\leq d$, $i\not\in\{p,p+1\}$ and
\[
\lim_{N\to+\infty}\left( \frac{X_p^N(0)}{N}, \frac{X_{p+1}^N(0)}{N}\right)=\left( (p+1)\beta -\rho,  \rho-p\beta\right)
\]
then, for the convergence in distribution, 
\[
\lim_{N\to+\infty} \left(\frac{X_0^N(N^{p-1}t)}{N},\frac{X_p^N(N^{p-1}t)}{N},\frac{X_{p+1}^N(N^{p-1}t)}{N}\right)=\left(\Phi_0(t),\Phi_p(t),\Phi_{p+1}(t)\right)
\]
where, for $t\geq 0$,
\[
\Phi_p(t)=(p+1)(\beta-\Phi_0(t))-\rho \text{ and } \Phi_{p+1}(t)= \rho-p(\beta-\Phi_0(t))
\]
and $\Phi_0(t)$ is the unique solution $y$ of the fixed point equation
\begin{equation}\label{decay2}
\left(1-\frac{y}{\beta-\rho/(p+1)}\right)^{\rho/(p(p+1))} e^y=\exp\left(-\lambda\frac{(p{-}1)!}{\rho^{p-1}} t\right).
\end{equation}
In particular,
\begin{equation}\label{eqlaux}
\lim_{t\to+\infty} \left(\Phi_0(t),\Phi_p(t),\Phi_{p+1}(t)\right)=\left(\beta-\frac{\rho}{p+1}, 0,\frac{\rho}{p+1} \right).
\end{equation}
\end{theorem}
\noindent
{\bf Remark.} 
Relation~\eqref{eqlaux} shows that a fraction $\beta{-}\rho/(p{+}1)$ of the files is lost asymptotically on the time scale $t\mapsto N^{p-1}t$.  The  corresponding asymptotic state consists then of  files which are either lost and, at the first order in $N$,  $\rho/(p{+}1)\cdot N$  files with $p{+}1$ copies.   This suggests that $\beta$ is changed to $\beta'{=}\rho/(p{+}1)$ and $p$ replaced by $p'{=}p{+}1$. Unfortunately, this is the case of equality $\beta'{=}p' \rho$ which is not covered by our theorem.  This suggests nevertheless the following evolution on the time scale $t\mapsto N^{q}t$, $p{-}1{\leq} q{\leq} d{-}2$,   for $t$ going to infinity, there remain $\rho/(q{+}2)N$ files alive with $q{+}2$. Some of the files are therefore lost and the number of copies of the remaining files is increasing, until the maximum number of copies is reached which is the framework of Section~\ref{AverageSec}.
\begin{proof}
The proofs use the same arguments as in the proof of Theorem~\ref{theodec} and of the above proposition. We give a quick overview of it. 
By using again the results of Propositions~\ref{prop1} and ~\ref{limit} and Relation~\eqref{id1}, one gets that, for $1\leq  k\leq p-2$, 
\[
\lim_{N\to\infty}\left( \int_0^{N^{p-1}t} 
\left[\frac{(k+1)!}{\rho^{k+1}} \frac{X_{k+1}^N(u)}{N^{k+1}}- \frac{k!}{\rho^{k}}  \frac{X_{k}^N(u)}{N^{k}}\right]\,\diff u\right)=0
\]
holds for the convergence in distribution. By summing up all these relations for $1\leq k\leq p-2$, one gets
\[
\lim_{N\to+\infty} \left(\frac{(p-1)!}{\rho^{p-1}} \frac{1}{N^{p-1}}\int_0^{N^{p-1}t}   X_{p-1}^N(u)\,\diff u 
-\frac{\mu}{\lambda}\int_0^{N^{p-1}t} \frac{X_1^N(u)}{N}\,\diff u\right)=0.
\]
From there one gets that
\[
\lim_{N\to+\infty} \left( \frac{X_0^N(u)}{N} -\frac{(p-1)!}{\rho^{p-1}} \frac{\lambda}{N^{p-1}}\int_0^{N^{p-1}t}   X_{p-1}^N(u)\,\diff u\right)=0.
\]
As in the proof of  Proposition~\ref{proppsi}, one can define a similar $(\Psi_f^N(t))$ and prove  the same stochastic averaging property associated to the coordinate $(X_p^N(t))$. The rest of the proof is then similar to the proof of the last proposition with $\beta$ replaced by $\beta-\phi(t)$ where $(\phi(t))$ is the limit of some converging subsequence of $(X_0^N(N^{p-1}t)/N)$. The convergence follows from the uniqueness of the fixed point equation satisfied by $(\phi(t))$. 
\end{proof}
\appendix
\section{Generalized Skorohod Problems}\label{SkoSec}
For the sake of self-containedness, this section presents quickly the more or less classical material necessary to state and prove the convergence results used in this paper.  The general theme concerns the rigorous definition of a solution of a stochastic differential equation constrained to stay in some domain and also the proof of the existence and uniqueness and regularity properties of such a solution. See Skorohod~\cite{Skorokhod}, Anderson and Orey~\cite{Anderson}, Chaleyat-Maurel and El~Karoui~\cite{Elkaroui} and, in a multi-dimensional context, Harrison and Reiman~\cite{Harrison} and Taylor and Williams~\cite{Taylor} and, in a more general context, Ramanan~\cite{Ramanan}. See Appendix~D of Robert~\cite{Robert} for a brief account.

We  first  recall  the classical  definition  of  Skorohod  problem in  dimension~$K$.  If
$z=(z_k)\in\R^K$,  one denotes  $\|z\|=|z_1|+|z_2|+\cdots+|z_k|$.  If  $(Z(t))=(Z_k(t))$ is
some function of  the set ${\cal D}(\R_+,\R^K)$ of c\`adl\`ag  functions defined on $\R_+$
and   $P$   is   a   $K\times   K$   non-negative  matrix,   the   couple   of   functions
$[(X(t)),(R(t))]=[((X_k(t))),((R_k(t)))]$ is said to be a solution of the Skorohod problem
associated to $(Z(t))$ and $P$ whenever
\begin{enumerate}
\item $X(t)=Z(t)+(I-P)\cdot  R(t)$, for all $t\geq 0$,
\item $X_k(t)\geq 0$, for all $t\geq 0$ and $1\leq k\leq d$,
\item For $1\leq k\leq K$, $t\to R_k(t)$ is non-decreasing, $R_k(0)=0$ and
\[
\int_{\R_+} X_k(t)\diff R_k(t)=0.
\]
\end{enumerate}
In the important case of dimension $1$, Conditions~(1) and~(3) are
\begin{enumerate}
\item $X(t)=Z(t)+R(t)$, for all $t\geq 0$,
\addtocounter{enumi}{+1}
\item $t\to R(t)$ is non-decreasing, $R(0)=0$ and
\[
\int_{\R_+} X(t)\diff R(t)=0.
\]
\end{enumerate}
See Chaleyat-Maurel and El~Karoui~\cite{Elkaroui} and, in a multi-dimensional context,
Harrison and Reiman~\cite{Harrison} and Taylor and Williams~\cite{Taylor}. See Appendix~D
of Robert~\cite{Robert} for a brief account. 
The generalization used in this paper corresponds to the case when $(Z(t))$ depends on
$(X(t))$. 
\begin{defi}[Generalized Skorohod Problem]\ \\
If $G: {\cal D}(\R_+,\R^K)\to {\cal D}(\R_+,\R^K)$ is a Borelian function and $P$ a non-negative
$K\times K$ matrix, $((X(t)),(R(t)))$ is a solution of the generalized Skorohod Problem
(GSP) associated to $G$ and $P$ if   $((X(t)),(R(t)))$ is the solution of the Skorohod Problem
associated to $G(X)$ and $P$, in particular, for all $t\geq 0$,  $$X(t)=G(X)(t)+(I-P)\cdot
R(t),$$
and 
\[
\int_{\R_+} X_k(t)\diff R_k(t)=0,\quad 1\leq k\leq K.
\]
\end{defi}
The classical Skorohod problem described above corresponds to the case when the functional
$G$ is constant and equal to $(Z(t))$. In dimension one, if one takes
\[
G(x)(t)=\int_0^t\sigma(x(u))\,\diff B(u) +\int_0^tm(x(u))\,\diff u,
\]
where  $(B(t))$ is a standard Brownian motion and $\sigma$ and $m$ are Lipschitz functions on
$\R$.  The first coordinate $(X(t))$ of a possible solution to the corresponding GSP can
be described as the solution of the SDE  
\[
\diff X(t)=\sigma(X(t))\,\diff B(t)+m(X(t))\,\diff t
\]
reflected at $0$.\\
\begin{prop}\label{GSPprop}
If $G: {\cal D}(\R_+,\R)\to {\cal D}(\R_+,\R)$ is such that, for any $T>0$, there exists a constant $C_T$
such that, for all  $(x(t))\in {\cal D}(\R_+,\R)$ and $0\leq t\leq T$,
\begin{equation}\label{Lip}
\sup_{0\leq s\leq t} \|G(x)(s)-G(y)(s)\|\leq C_T \int_0^t \|x(u)-y(u)\|\,\diff u
\end{equation}
and if the matrix $P$ is nilpotent, then there exists a unique solution to the
generalized Skorohod problem associated to the functional $G$ and the matrix $P$. 
\end{prop}
\begin{proof}
Define the sequence $(X_N(t))$ by induction $(X^0(t),R^0(t))=0$ and, for $N\geq 1$,  $(X^{N+1},R^{N+1})$ is the
solution of the Skorohod problem  (SP) associated to $G(X^{N})$, in particular,
\[
X^{N+1}(t)=F\left(X^N\right)(t)+R^{N+1}(t) \text{ and } \int_{\R_+} X^{N+1}(u)\diff R^{N+1}(u)=0.
\]
The existence of such a solution is a consequence of a result of Harrison and Reiman~\cite{Harrison}.
Fix $T>0$. The Lipschitz property of the solutions of a classical Skorohod problem, see Proposition~D.4 of
Robert~\cite{Robert}, gives the existence of some constant $K_T$ such that, for all
$N\geq 1$ and $0\leq t\leq T$, 
\[
\left\|X^{N+1}-X^{N}\right\|_{\infty,t}\leq K_T \left\|F\left(X^{N}\right)-F\left(X^{N-1}\right)\right\|_{\infty,t},
\]
where $\|h\|_{\infty,T}= \sup\{\|h(s)\|:0\leq s\leq T\}$. From Relation~\eqref{Lip}, this
implies that
\[
\left\|X^{N+1}-X^{N}\right\|_{\infty,t}\leq \alpha\int_0^{t} \left\|X^{N}-X^{N-1}\right\|_{\infty,u}\,\diff u,
\]
with $\alpha=K_TC_T$. The iteration of the last relation yields the inequality
\[
\left\|X^{N+1}-X^{N}\right\|_{\infty,t}\leq \frac{(\alpha t)^N}{N!} \int_0^{t}
\left\|X^{1}\right\|_{\infty,u}\,\diff u, \quad 0\leq t\leq T.
\]
One concludes that the sequence $(X^N(t))$ is converging uniformly on compact sets and
consequently the same is true for the sequence $(R^N(t))$. Let $(X(t))$ and $(R(t))$ be the
limit of these sequences. By continuity of the SP, the couple $((X(t)), (R(t)))$ is the
solution of the SP associated to $G(X)$, and hence a solution of the GSP associated to $F$. 

{\bf Uniqueness.} If $(Y(t))$ is another solution of the GSF associated to $F$. In the same way
as before, one gets by induction, for $0\leq t\leq T$,
\[
\left\|X-Y\right\|_{\infty,t}\leq \frac{(\alpha t)^N}{N!} \int_0^{t} \left\|X-Y\right\|_{\infty,u}\,\diff u,
\]
and by letting $N$ go to infinity, one concludes that $X=Y$. The proposition is proved. 
\end{proof}

\providecommand{\bysame}{\leavevmode\hbox to3em{\hrulefill}\thinspace}
\providecommand{\MR}{\relax\ifhmode\unskip\space\fi MR }
\providecommand{\MRhref}[2]{%
  \href{http://www.ams.org/mathscinet-getitem?mr=#1}{#2}
}
\providecommand{\href}[2]{#2}

\end{document}